\documentclass{amsart}
\usepackage[utf8]{inputenc}

\usepackage[margin=1.5in]{geometry}

\usepackage{tikz}
\usepackage{enumerate}
\usepackage{xcolor}
\usepackage{graphicx}
\usepackage{microtype}
\usepackage{caption}
\usepackage{subcaption}
\usepackage{hyperref}
\hypersetup{
    colorlinks=true,
    linkcolor=blue,     
    urlcolor=blue,
    citecolor=blue
    }

\usepackage{thmtools}
\usepackage{thm-restate}

\usetikzlibrary{calc,
positioning,
decorations.pathmorphing, decorations.pathreplacing, decorations.shapes,
arrows.meta,
bending}

\definecolor{0}{RGB}{250,0,0}
\definecolor{1}{RGB}{0,10,255}
\definecolor{2}{RGB}{220,190,0}
\definecolor{3}{RGB}{0,170,130}

\definecolor{midnightblue}{HTML}{191970}
\definecolor{darkgreen}{HTML}{006400}
\definecolor{gold}{HTML}{ffd700}
\definecolor{lime}{HTML}{00ff00}
\definecolor{aqua}{HTML}{00ffff}
\definecolor{fushia}{HTML}{ff00ff}
\definecolor{lightpink}{HTML}{ffb6c1}

\definecolor{orange}{HTML}{ffa500}
\definecolor{darkcyan}{HTML}{008b8b}
\definecolor{deeppink}{HTML}{ff1493}

\definecolor{yellowish}{HTML}{fcc400}
\definecolor{purplely}{HTML}{ac00e8}

\theoremstyle{plain}
\newtheorem{thm}{Theorem}
\newtheorem{prop}[thm]{Proposition}
\newtheorem{lem}[thm]{Lemma}
\newtheorem{cor}[thm]{Corollary}
\newtheorem{conj}{Conjecture}

\newtheorem{lemma}[thm]{Lemma}

\theoremstyle{definition}

\newtheorem{prob}{Problem}

\theoremstyle{remark}

\newcommand{\topo}{\mathrm{top}}
\newcommand{\odd}{\mathrm{odd}}
\newcommand{\join}{\texttt{join}}
\newcommand{\ins}{\texttt{insert}}
\newcommand{\meet}{\texttt{meet}}
\newcommand{\la}{\mathrm{la}}
\newcommand{\iso}{\mathrm{iso}}

\title{Path Odd-Covers of Graphs}
\author[]{Steffen~Borgwardt \and Calum~Buchanan \and Eric~Culver \and Bryce~Frederickson \and Puck~Rombach \and  Youngho~Yoo}

\begin{document}
\begin{abstract}
We introduce and study ``path odd-covers'', a weakening of Gallai’s path decomposition problem and a strengthening of the linear arboricity problem. The ``path odd-cover number'' $p_2(G)$ of a graph $G$ is the minimum cardinality of a collection of paths whose vertex sets are contained in $V(G)$ and whose symmetric difference of edge sets is $E(G)$.

We prove an upper bound on $p_2(G)$ in terms of the maximum degree $\Delta$ and the number of odd-degree vertices $v_{\mathrm{odd}}$ of the form $\max\left\{{v_{\mathrm{odd}}}/{2}, 2\left\lceil {\Delta}/{2}\right \rceil\right\}$. This bound is only a factor of $2$ from a rather immediate lower bound of the form $\max \left\{ {v_{\mathrm{odd}} }/{2} , \left\lceil {\Delta}/{2}\right\rceil \right\}$.
We also investigate some natural relaxations of the problem which highlight the connection between the path odd-cover number and other well-known graph parameters.
For example, when allowing for subdivisions of $G$, the previously mentioned lower bound is always tight except in some trivial cases. 
Further, a relaxation that allows for the addition of isolated vertices to $G$ leads to a match with the linear arboricity when $G$ is Eulerian.
Finally, we transfer our observations to establish analogous results for cycle odd-covers.
\end{abstract}


\maketitle

\section{Introduction}\label{sec:intro}

Gallai conjectured that no more than $\lceil n/2 \rceil$ edge-disjoint paths are needed to decompose the edges of any connected graph on $n$ vertices. This would answer a question posed by Erdős. Both are posed in~\cite{lovasz1968covering}. The conjecture remains open in general, but it has been proven for many classes of graphs, such as graphs whose even-degree vertices induce a forest~\cite{pyber1985erdHos}, planar graphs~\cite{blanche2021gallai}, graphs with maximum degree at most $5$~\cite{bonamy2019gallai}, and graphs with maximum degree $6$ in which the vertices of maximum degree form an independent set (with some exceptions)~\cite{chu2021gallai}.
A number of related invariants have been introduced in the study of Gallai's conjecture.
Lov\'{a}sz showed that at most $\lfloor n/2 \rfloor$ edge-disjoint paths and cycles suffice to decompose the edges of any $n$-vertex graph~\cite{lovasz1968covering}, which also implies that $n$ edge-disjoint paths suffice.
This bound was subsequently improved to $\lfloor 2n/3\rfloor$ independently by Yan \cite{yan1998} and by Dean and Kouider \cite{dean2000}.
Fan resolved the ``covering'' version of the problem by showing that $\lceil n/2 \rceil$ (not necessarily edge-disjoint) paths suffice to cover the edges of any connected $n$-vertex graph~\cite{fan2002subgraph}, as conjectured by Chung \cite{chung1980coverings}.

Similar problems have also been considered for decomposing graphs into cycles. Erd\H{o}s and Gallai conjectured that the edges of any $n$-vertex graph can be decomposed into $O(n)$ cycles and edges (see \cite{erdHos1983some}). Equivalently, this says that the edges of any Eulerian graph can be decomposed into $O(n)$ cycles.  Haj\'os made the stronger conjecture that $\lfloor n/2 \rfloor$ cycles are always enough (see \cite{lovasz1968covering}), which would actually imply that $\lfloor (n-1)/2 \rfloor$ cycles suffice, as pointed out by Dean \cite{dean1986smallest}. By greedily removing large cycles from the graph, it is not hard to show that any Eulerian graph can be decomposed into $O(n \log n)$ cycles. The first improvement on this bound was by Conlon, Fox, Sudakov, who proved the bound of $O(n \log \log n)$ \cite{conlon2014cycle}. More recently, Buci\'c and Montgomery improved this to $O(n \log^\star n)$, where $\log^\star n$ is the iterated logarithm function \cite{bucic2022towards}. ``Covering'' versions of these problems have also been considered, and entirely resolved. Fan proved that the edges of any Eulerian graph $G$ on $n$ vertices can be covered by $\lfloor (n-1)/2 \rfloor$ cycles from $G$; in fact, this covering can be taken so that each edge is covered an odd number of times \cite{fan2003covers}. Similarly, Pyber proved that any $n$-vertex graph can be covered by covering by $n-1$ cycles and edges \cite{pyber1985erdHos}.

In this work, we introduce a variant of these path and cycle decomposition problems. 
A \emph{path odd-cover} of a graph $G = (V,E)$ is a collection of paths in the complete graph on $V$ ({\em i.e.}, paths whose vertex sets are contained in $V$) whose symmetric difference of edge sets is $E$; that is, every edge of $G$ appears in an odd number and every nonedge appears in an even number of the paths in the collection.
The \emph{path odd-cover number} $p_2(G)$ of $G$ is the minimum cardinality of a path odd-cover of $G$.
If $G$ is Eulerian, we similarly define a \emph{cycle odd-cover} of $G$ to be a collection of cycles whose symmetric difference of edge sets is $E$. We call the minimum cardinality of such a collection the \emph{cycle odd-cover number} of $G$, which we denote by $c_2(G)$.
The subscript ``$2$" indicates the relation to linear algebra over $\mathbb{F}_2$; a path (or a cycle) odd-cover can be equivalently defined as a collection of graphs on $V$ consisting of a single path (or cycle), and possibly some isolated vertices, whose adjacency matrices sum to the adjacency matrix of $G$ over $\mathbb{F}_2$.

The path odd-cover number is closely related to both the {\em path decomposition number} $p(G)$ and the {\em linear arboricity} $\la(G)$ of a graph $G$, which are the minimum cardinalities of a collection of edge-disjoint paths or linear forests which decompose $E(G)$, respectively. Since a linear forest (and hence a path) has maximum degree $2$, it is immediate that for a graph $G$ with maximum degree $\Delta$, $p(G)$ and $\la(G)$ are each bounded below by $\Delta / 2$.
Another long-standing conjecture, posed in~\cite{akiyama1980}, is that the latter bound is nearly tight, i.e., that $\la(G) \leq \left\lceil \frac{\Delta + 1}{2} \right\rceil$.
Alon~\cite{alon1988} proved that this upper bound holds asymptotically as $\Delta \to \infty$, and the best current such bound is $\Delta / 2 + O(\Delta^{2/3 - \alpha})$ for some absolute constant $\alpha > 0$~\cite{ferber2020towards}.

The notion of an odd-cover has appeared in the past in different contexts. It can be traced back to a problem of Babai and Frankl from 1988 \cite{babai1988linear} which asks for the minimum number of complete bipartite subgraphs of $K_n$ whose symmetric difference is $K_n$. This question was motivated as the binary field adaptation of the celebrated linear algebraic proof of Graham and Pollak \cite{graham1971,graham1972} that one needs at least $n-1$ edge-disjoint complete bipartite graphs to decompose $K_n$. Babai and Frankl's question was generalized in~\cite{buchanan2022odd} to finding the minimum cardinality of a biclique odd-cover of an arbitrary graph $G$.
Relatedly, one can replace complete bipartite graphs by complete graphs to obtain the minimum cardinality of a clique odd-cover of $G$. Clique odd-covers are studied under the name {\em subgraph complementation systems} in~\cite{buchanan2022subgraph}; these were motivated by a question posed by Vatter~\cite{317716} on representing a graph $G$ as a sum of cliques modulo 2, and they are equivalent to finding {\em orthogonal representations} of $G$ over $\mathbb{F}_2$~\cite{lovasz1979shannon}. 

The idea of using nonedges of $G$ also arises naturally in the study of the combinatorial diameter of Birkhoff polytopes \cite{br-74}, which represent one-to-one assignments, and the more general partition polytopes \cite{b-13,bv-19a}. The difference of two partitions of the same data set can be represented as a so-called {\em difference graph}, where the vertices correspond to the partition parts and each edge corresponds to a data item that has to be moved from one part to another. A walk in the skeleton of a partition polytope can be represented as a sequence of paths and cycles to ``delete'' such a graph. To minimize the length of such a sequence, one allows for the use of nonedges, {\em i.e.}, for a temporary misplacement of items.

\subsection{Contributions}

In this paper we devise upper and lower bounds on $p_2(G)$ based on other graph parameters of $G$. Chief among these parameters are the maximum degree of $G$, $\Delta(G)$, and the number of odd-degree vertices in $G$, $v_\odd(G)$. We also relate $p_2(G)$ to other parameters, such as the linear arboricity, $\la(G)$, and the path decomposition number, $p(G)$. Further, we define two variants of the path odd-cover number, $p_{2,\topo}(G)$ and $p_{2,\iso}(G)$, which give lower bounds on $p_2(G)$ and which have interesting properties of their own. The relationship between these parameters is summarized in Figure \ref{fig:bound diagram}.

\begin{figure}[ht]
  \centering
  \tikzstyle{every node}=[circle, draw, fill=black, inner sep=0pt, minimum width=4pt]
  \tikzstyle{every path}=[thick]
  \tikzstyle{label}=[draw=none, fill=none]
  \begin{tikzpicture}
    \node(p) at (2,4) {};
    \node(p2) at (1,3) {};
    \node(pc) at (3,2) {};
    \node(p2iso) at (1,2) {};
    \node(p2top) at (1,1) {};
    \node(la) at (3,1) {};
    \node(a) at (4,0) {};
    \node(max) at (2,0) {};
    \node(odd) at (0,0) {};
    
    \draw (p)--(p2)--(p2iso)--(p2top);
    \draw (p)--(pc)--(la)--(a);
    \draw (p2iso)--(la);
    \draw (odd) -- (p2top) -- (max) -- (la) -- (a);
    
    \node[label](hoffset) at (0.6,0) {};
    \node[label](voffset) at (0,-0.5) {};
    \node[label](l.p) at ($(p) - (hoffset) + (0,0.1)$) {$p(G)$};
    \node[label](l.p2) at ($(p2) - (hoffset)$) {$p_2(G)$};
    \node[label](l.pc) at ($(pc) + 1*(hoffset)$) {$p_c(G)$};
    \node[label](l.p2iso) at ($(p2iso) - 1.5*(hoffset)$) {$p_{2,\iso}(G)$};
    \node[label](l.p2top) at ($(p2top) - 1.5*(hoffset)$) {$p_{2,\topo}(G)$};
    \node[label](l.odd) at ($(odd) + (voffset)$) {$\frac{v_\odd(G)}{2}$};
    \node[label](l.max) at ($(max) + (voffset)$) {$\Big\lceil \frac{\Delta(G)}{2} \Big\rceil$};
    \node[label](l.la) at ($(la) + (hoffset)$) {$\la(G)$};
    \node[label](l.a) at ($(a) + (voffset)$) {$a(G)$};
  \end{tikzpicture}
  \caption{A Hasse diagram of the various parameters discussed in this paper. For each edge, the parameter at the lower vertex is a lower bound for the parameter at the upper vertex.
  \label{fig:bound diagram}}
\end{figure}
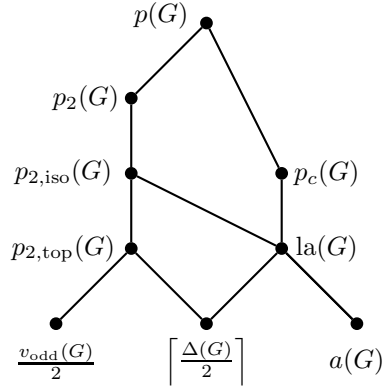  

Some elements of this figure are immediate. For example, every path decomposition of $G$ is a path odd-cover, and thus $p_2(G) \leq p(G)$.
Also, if we have a path odd-cover of $G$ with $k$ paths, then {\em (i)}\;each path contains at most two odd-degree vertices, so that $v_\odd(G) \leq 2k$, and {\em (ii)}\;each path has maximum degree $2$, so that $\Delta(G) \leq 2k$.
In other words, we have
\begin{equation}\label{eq:vodd,maxdegree_lower}
p_2(G) \geq \max{ \left\{ \frac{v_\odd (G)}{2} , \left\lceil\frac{\Delta(G)}{2}\right\rceil \right\} }.
\end{equation}

Also included in Figure~\ref{fig:bound diagram} is the {\em path covering number} of $G$, $p_c(G)$, which is the minimum number of paths in $G$ needed to cover every edge at least once~\cite{chung1980coverings}.
While $p_c(G)$ might appear to be related to $p_2(G)$, and while both are bounded below by $\la(G)$ and above by $p(G)$, these parameters can be arbitrarily far apart in either direction. Figures~\ref{fig:cyclecover} and \ref{fig:pclessp2} depict examples.

\begin{figure}[ht]
\centering        
\begin{tikzpicture}[vertices/.style={draw, fill=black, circle, inner sep=0pt, minimum size = 4pt, outer sep=0pt}, r_edge/.style={draw=red,line width= 1,>=latex,red}, b_edge/.style={draw=blue,line width= 1,>=latex,blue}, k_edge/.style={draw=black,line width= 1,>=latex,black}, scale=1]
\path[use as bounding box] (-0.5,-1) rectangle (9,2.5);
\node[vertices] (w_1) at (-.25,1) {};
\node[vertices] (w_2) at (.5,1.75) {};
\node[vertices] (w_3) at (1.25, 1) {};
\node[vertices] (w_4) at (.5, 0.25) {};

\node[vertices] (c_1) at (2.0,1.1) {};
\node[vertices] (c_2) at (3.0, 2) {};
\node[vertices] (c_3) at (4.0, 1.1) {};
\node[vertices] (c_4) at (3.6, 0) {};
\node[vertices] (c_5) at (2.4, 0) {};

\node[vertices] (d_1) at (5,1.1) {};
\node[vertices] (d_2) at (6, 2) {};
\node[vertices] (d_3) at (7, 1.1) {};
\node[vertices] (d_4) at (6.6, 0) {};
\node[vertices] (d_5) at (5.4, 0) {};

\node[vertices] (z_1) at (8,0.4) {};
\node[vertices] (z_2) at (8,1.6) {};
\node[vertices] (z_3) at (9, 1){};

\foreach \to/\from in {w_1/w_2, w_2/w_3, w_3/w_4,
    c_1/c_2, c_2/c_3, c_3/c_4, c_5/c_1,
    d_1/d_2, d_2/d_3, d_3/d_4, d_5/d_1,
    z_1/z_2, z_2/z_3}
\draw[k_edge]  (\to)--(\from);
\foreach \to/\from in {w_4/w_1,
    c_4/c_5,
    d_4/d_5,
    z_3/z_1}
\draw[k_edge]  (\to)--(\from);

\node at (-1.5,1) {$G$};
\end{tikzpicture}

\begin{tikzpicture}[vertices/.style={draw, fill=black, circle, inner sep=0pt, minimum size = 4pt, outer sep=0pt}, r_edge/.style={draw=red,line width= 1,>=latex,red}, b_edge/.style={draw=blue,line width= 1,>=latex,blue}, k_edge/.style={draw=black,line width= 1,>=latex,black}, scale=1]
\path[use as bounding box] (-0.5,-1) rectangle (9,2.5);
\node[vertices] (w_1) at (-.25,1) {};
\node[vertices] (w_2) at (.5,1.75) {};
\node[vertices] (w_3) at (1.25, 1) {};
\node[vertices] (w_4) at (.5, 0.25) {};

\node[vertices] (c_1) at (2.0,1.1) {};
\node[vertices] (c_2) at (3.0, 2) {};
\node[vertices] (c_3) at (4.0, 1.1) {};
\node[vertices] (c_4) at (3.6, 0) {};
\node[vertices] (c_5) at (2.4, 0) {};

\node[vertices] (d_1) at (5,1.1) {};
\node[vertices] (d_2) at (6, 2) {};
\node[vertices] (d_3) at (7, 1.1) {};
\node[vertices] (d_4) at (6.6, 0) {};
\node[vertices] (d_5) at (5.4, 0) {};

\node[vertices] (z_1) at (8,0.4) {};
\node[vertices] (z_2) at (8,1.6) {};
\node[vertices] (z_3) at (9, 1){};

\foreach \to/\from in {w_1/w_2, w_2/w_3, w_3/w_4,
    c_1/c_2, c_2/c_3, c_3/c_4, c_5/c_1,
    d_1/d_2, d_2/d_3, d_3/d_4, d_5/d_1,
    z_1/z_2, z_2/z_3}
\draw[b_edge]  (\to)--(\from);
\foreach \to/\from in {w_4/w_1,
    c_4/c_5,
    d_4/d_5,
    z_3/z_1}
\draw[r_edge]  (\to)--(\from);

\path [b_edge] (c_5) edge[bend right=30, looseness=0] (w_4);
\path [b_edge] (d_5) edge[bend right=30, looseness=0] (c_4);
\path [b_edge] (z_1) edge[bend right=30, looseness=0] (d_4);
\path [r_edge] (c_5) edge[bend left=30, looseness=0] (w_4);
\path [r_edge] (d_5) edge[bend left=30, looseness=0] (c_4);
\path [r_edge] (z_1) edge[bend left=30, looseness=0] (d_4);

\node at (-1.5,1) {$p_2(G)$};

\end{tikzpicture}

\begin{tikzpicture}[vertices/.style={draw, fill=black, circle, inner sep=0pt, minimum size = 4pt, outer sep=0pt}, r_edge/.style={draw=red,line width= 1,>=latex,red}, b_edge/.style={draw=blue,line width= 1,>=latex,blue}, k_edge/.style={draw=black,line width= 1,>=latex,black}, m_edge/.style={draw=midnightblue,line width= 1,>=latex,midnightblue}, d_edge/.style={draw=darkgreen,line width= 1,>=latex,darkgreen}, a_edge/.style={draw=aqua,line width= 1,>=latex,aqua}, l_edge/.style={draw=orange,line width= 1,>=latex,orange}, f_edge/.style={draw=fushia,line width= 1,>=latex,fushia}, i_edge/.style={draw=lime,line width= 1,>=latex,lime}, scale=1]
\path[use as bounding box] (-0.5,-1) rectangle (9,2.5);
\node[vertices] (w_1) at (-.25,1) {};
\node[vertices] (w_2) at (.5,1.75) {};
\node[vertices] (w_3) at (1.25, 1) {};
\node[vertices] (w_4) at (.5, 0.25) {};

\node[vertices] (c_1) at (2.0,1.1) {};
\node[vertices] (c_2) at (3.0, 2) {};
\node[vertices] (c_3) at (4.0, 1.1) {};
\node[vertices] (c_4) at (3.6, 0) {};
\node[vertices] (c_5) at (2.4, 0) {};

\node[vertices] (d_1) at (5,1.1) {};
\node[vertices] (d_2) at (6, 2) {};
\node[vertices] (d_3) at (7, 1.1) {};
\node[vertices] (d_4) at (6.6, 0) {};
\node[vertices] (d_5) at (5.4, 0) {};

\node[vertices] (z_1) at (8,0.4) {};
\node[vertices] (z_2) at (8,1.6) {};
\node[vertices] (z_3) at (9, 1){};

\foreach \to/\from in {w_1/w_2, w_2/w_3, w_3/w_4}
\draw[b_edge]  (\to)--(\from);
\foreach \to/\from in {c_1/c_2, c_2/c_3, c_3/c_4, c_5/c_1}
\draw[i_edge]  (\to)--(\from);
\foreach \to/\from in {d_1/d_2, d_2/d_3, d_3/d_4, d_5/d_1}
\draw[l_edge]  (\to)--(\from);
\foreach \to/\from in {z_1/z_2, z_2/z_3}
\draw[f_edge]  (\to)--(\from);

\draw[r_edge] (w_4) -- (w_1);
\draw[d_edge] (c_4) -- (c_5);
\draw[a_edge] (d_4) -- (d_5);
\draw[m_edge] (z_3) -- (z_1);

\node at (-1.5,1) {$p_c(G)$};
\end{tikzpicture}
\caption{A set of $k$ vertex-disjoint cycles has $p_2(G) = 2 \ll p_c(G) = p(G) = 2k$.}
\label{fig:cyclecover}
\end{figure}
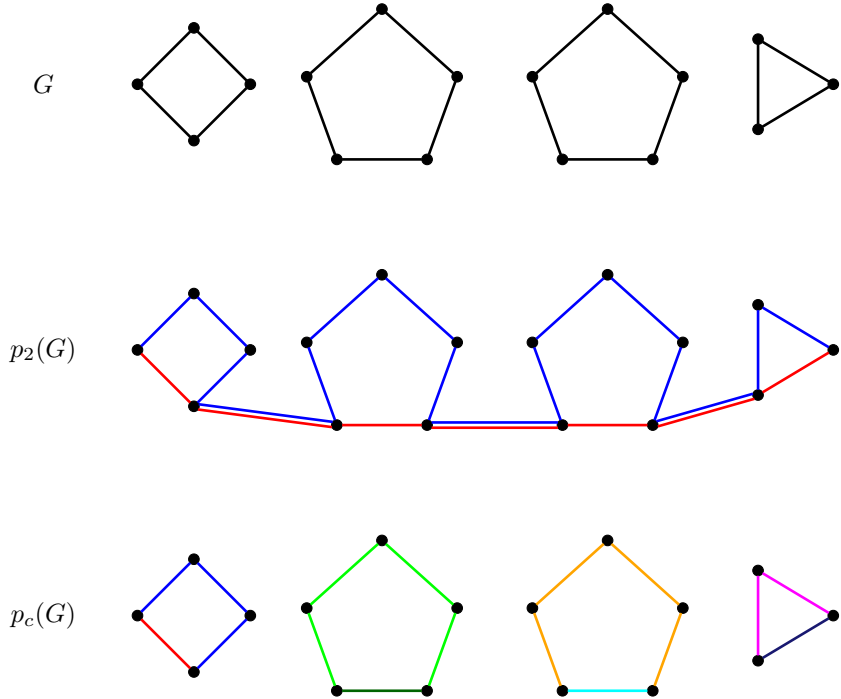

\begin{figure}[ht]
         \centering        
\begin{tikzpicture}[vertices/.style={draw, fill=black, circle, inner sep=0pt, minimum size = 4pt, outer sep=0pt}, r_edge/.style={draw=red,line width= 1,>=latex,red}, b_edge/.style={draw=blue,line width= 1,>=latex,blue}, k_edge/.style={draw=black,line width= 1,>=latex,black}, scale=1]
\path[use as bounding box] (-1.5,-1) rectangle (10,2.5);
\node[vertices] (w_1) at (-.25,1) {};
\node[vertices] (w_2) at (.5,1.75) {};
\node[vertices] (w_3) at (1.25, 1) {};
\node[vertices] (w_4) at (.5, 0.25) {};

\node[vertices] (c_1) at (2.0,1.1) {};
\node[vertices] (c_2) at (3.0, 2) {};
\node[vertices] (c_3) at (4.0, 1.1) {};
\node[vertices] (c_4) at (3.6, 0) {};
\node[vertices] (c_5) at (2.4, 0) {};

\node[vertices] (d_1) at (5,1.1) {};
\node[vertices] (d_2) at (6, 2) {};
\node[vertices] (d_3) at (7, 1.1) {};
\node[vertices] (d_4) at (6.6, 0) {};
\node[vertices] (d_5) at (5.4, 0) {};

\node[vertices] (z_1) at (8,0.4) {};
\node[vertices] (z_2) at (8,1.6) {};
\node[vertices] (z_3) at (9, 1){};

\node[vertices] (f_1) at (-1, 1.8){};
\node[vertices] (f_2) at (9.75, 1.8){};

\foreach \to/\from in {w_1/w_2, w_2/w_3, w_3/w_4,
    c_1/c_2, c_2/c_3, c_3/c_4, c_5/c_1,
    d_1/d_2, d_2/d_3, d_3/d_4, d_5/d_1,
    z_1/z_2, z_2/z_3}
\draw[k_edge]  (\to)--(\from);
\foreach \to/\from in {w_4/w_1,
    c_4/c_5,
    d_4/d_5,
    z_3/z_1}
\draw[k_edge]  (\to)--(\from);

\draw[k_edge] (w_4) -- (c_5);
\draw[k_edge] (c_4) -- (d_5);
\draw[k_edge] (d_4) -- (z_1);

\draw[k_edge] (f_1) -- (w_1);
\draw[k_edge] (f_2) -- (z_3);

\node at (-1.5,1) {$G$};
\end{tikzpicture}

\begin{tikzpicture}[vertices/.style={draw, fill=black, circle, inner sep=0pt, minimum size = 4pt, outer sep=0pt}, r_edge/.style={draw=red,line width= 1,>=latex,red}, b_edge/.style={draw=blue,line width= 1,>=latex,blue}, k_edge/.style={draw=black,line width= 1,>=latex,black}, m_edge/.style={draw=aqua,line width= 1,>=latex,aqua}, d_edge/.style={draw=midnightblue,line width= 1,>=latex,midnightblue}, a_edge/.style={draw=darkgreen,line width= 1,>=latex,darkgreen}, l_edge/.style={draw=orange,line width= 1,>=latex,orange}, f_edge/.style={draw=fushia,line width= 1,>=latex,fushia}, i_edge/.style={draw=lime,line width= 1,>=latex,lime}, scale=1]
\path[use as bounding box] (-1.5,-1) rectangle (10,2.5);
\node[vertices] (w_1) at (-.25,1) {};
\node[vertices] (w_2) at (.5,1.75) {};
\node[vertices] (w_3) at (1.25, 1) {};
\node[vertices] (w_4) at (.5, 0.25) {};

\node[vertices] (c_1) at (2.0,1.1) {};
\node[vertices] (c_2) at (3.0, 2) {};
\node[vertices] (c_3) at (4.0, 1.1) {};
\node[vertices] (c_4) at (3.6, 0) {};
\node[vertices] (c_5) at (2.4, 0) {};

\node[vertices] (d_1) at (5,1.1) {};
\node[vertices] (d_2) at (6, 2) {};
\node[vertices] (d_3) at (7, 1.1) {};
\node[vertices] (d_4) at (6.6, 0) {};
\node[vertices] (d_5) at (5.4, 0) {};

\node[vertices] (z_1) at (8,0.4) {};
\node[vertices] (z_2) at (8,1.6) {};
\node[vertices] (z_3) at (9, 1){};

\foreach \to/\from in {w_1/w_2, w_2/w_3, w_3/w_4}
\draw[b_edge]  (\to)--(\from);
\foreach \to/\from in {c_1/c_2, c_2/c_3, c_3/c_4, c_5/c_1}
\draw[m_edge]  (\to)--(\from);
\foreach \to/\from in {d_1/d_2, d_2/d_3, d_3/d_4, d_5/d_1}
\draw[a_edge]  (\to)--(\from);
\foreach \to/\from in {z_1/z_2, z_2/z_3}
\draw[l_edge]  (\to)--(\from);

\foreach \to/\from in {w_4/w_1,
    w_4/c_5,
    c_4/c_5,
    c_4/d_5,
    d_4/d_5,
    d_4/z_1,
    z_3/z_1}
\draw[r_edge]  (\to)--(\from);

\node[vertices] (f_1) at (-1, 1.8){};
\node[vertices] (f_2) at (9.75, 1.8){};

\draw[r_edge] (f_1) -- (w_1);
\draw[r_edge] (f_2) -- (z_3);

\node at (-1.5,1) {$p_2(G)$};
\end{tikzpicture}

\begin{tikzpicture}[vertices/.style={draw, fill=black, circle, inner sep=0pt, minimum size = 4pt, outer sep=0pt}, r_edge/.style={draw=red,line width= 1,>=latex,red}, b_edge/.style={draw=blue,line width= 1,>=latex,blue}, k_edge/.style={draw=black,line width= 1,>=latex,black}, scale=1]
\path[use as bounding box] (-1.5,-1) rectangle (10,2.5);
\node[vertices] (w_1) at (-.25,1) {};
\node[vertices] (w_2) at (.5,1.75) {};
\node[vertices] (w_3) at (1.25, 1) {};
\node[vertices] (w_4) at (.5, 0.25) {};

\node[vertices] (c_1) at (2.0,1.1) {};
\node[vertices] (c_2) at (3.0, 2) {};
\node[vertices] (c_3) at (4.0, 1.1) {};
\node[vertices] (c_4) at (3.6, 0) {};
\node[vertices] (c_5) at (2.4, 0) {};

\node[vertices] (d_1) at (5,1.1) {};
\node[vertices] (d_2) at (6, 2) {};
\node[vertices] (d_3) at (7, 1.1) {};
\node[vertices] (d_4) at (6.6, 0) {};
\node[vertices] (d_5) at (5.4, 0) {};

\node[vertices] (z_1) at (8,0.4) {};
\node[vertices] (z_2) at (8,1.6) {};
\node[vertices] (z_3) at (9, 1){};

\foreach \to/\from in {w_1/w_2, w_2/w_3, w_3/w_4,
    c_1/c_2, c_2/c_3, c_3/c_4, c_5/c_1,
    d_1/d_2, d_2/d_3, d_3/d_4, d_5/d_1,
    z_1/z_2, z_2/z_3}
\draw[b_edge]  (\to)--(\from);
\foreach \to/\from in {w_4/w_1,
    c_4/c_5,
    d_4/d_5,
    z_3/z_1}
\draw[r_edge]  (\to)--(\from);

\path [b_edge] (c_5) edge[bend right=30, looseness=0] (w_4);
\path [b_edge] (d_5) edge[bend right=30, looseness=0] (c_4);
\path [b_edge] (z_1) edge[bend right=30, looseness=0] (d_4);
\path [r_edge] (c_5) edge[bend left=30, looseness=0] (w_4);
\path [r_edge] (d_5) edge[bend left=30, looseness=0] (c_4);
\path [r_edge] (z_1) edge[bend left=30, looseness=0] (d_4);

\node[vertices] (f_1) at (-1, 1.8){};
\node[vertices] (f_2) at (9.75, 1.8){};

\draw[r_edge] (f_1) -- (w_1);
\draw[r_edge] (f_2) -- (z_3);

\node at (-1.5,1) {$p_c(G)$};

\end{tikzpicture}
         \caption{A set of $k$ vertex-disjoint cycles attached to a path at consecutive internal vertices, as above, has $p_2(G) = k+1 \gg p_c(G) = 2$.}
         \label{fig:pclessp2}
\end{figure}
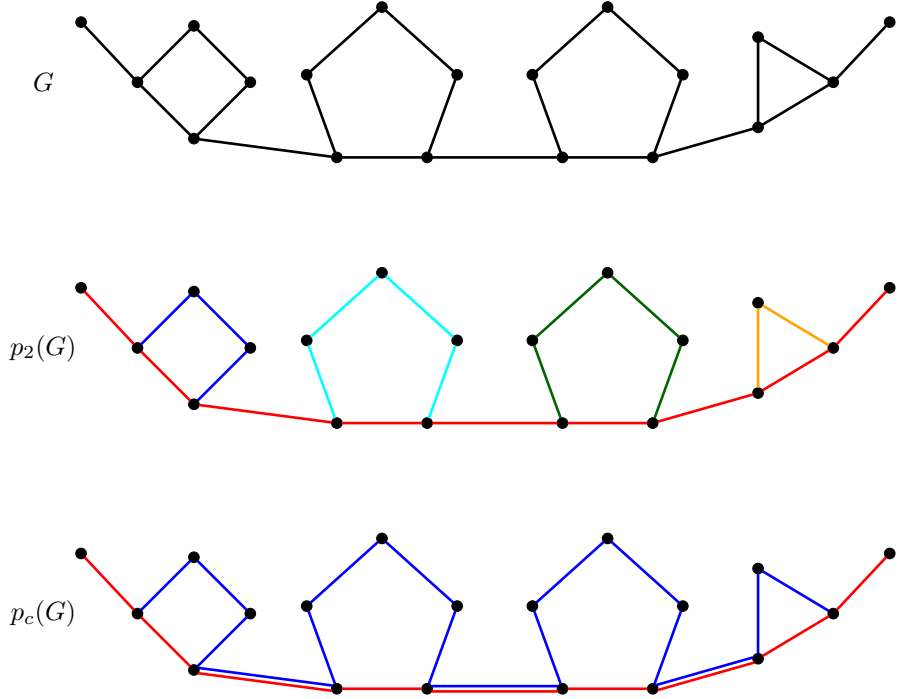

In Section~\ref{subsec:lowertopo}, we introduce $p_{2,\topo}(G)$, which is the minimum value of $p_2(H)$ over all subdivisions $H$ of $G$. Informally, this parameter considers the topological question of what kinds of ``shapes'' can be expressed as the symmetric difference of $k$ paths in space. Initially, our hope in investigating $p_{2,\topo}(G)$ was to see what other kinds of topological obstructions there are to path odd-covering besides vertex degrees. However, we discovered that there are none, except the trivial case where $G$ is the disjoint union of cycles and one path.
\begin{restatable}{thm}{Topological}\label{thm:topological}
If $G$ is not the disjoint union of at least one cycle with at most one path, then
\[p_{2,\topo}(G) = \max\left\{\frac{v_{\odd}(G)}{2}, \left\lceil \frac{\Delta(G)}{2}\right \rceil\right\}.\]
\end{restatable}

Another immediate lower bound for $p_2(G)$ is the linear arboricity of $G$. Indeed, if we have a path odd-cover of $G$ with $k$ paths, then by deleting duplicated edges in pairs we obtain a decomposition of $E(G)$ into $k$ edge-disjoint linear forests, and hence 
\begin{equation}\label{eq:la_lower}
    p_2(G)\geq \la(G).
\end{equation}

Moreover, linear arboricity is lower bounded by \emph{arboricity}, the number $a(G)$ of edge-disjoint forests needed to cover the edges of $G$. 
A classical result of Nash-Williams~\cite{Nash-Williams64} gives an exact formula for $a(G)$ in terms of the maximum edge-density of subgraphs of $G$. This gives us a further lower bound on $p_2(G)$:
\begin{equation}\label{eq:density_lower}
p_2(G) \geq a(G) = \max_{H \subseteq G} \frac{e(H)}{v(H)-1}, \end{equation}
where $e(H)$ and $v(H)$ denote the number of edges and the number of vertices in $H$, respectively.

In Section~\ref{subsec:loweriso}, we introduce $p_{2, \iso}(G)$, which is the minimum of $p_2(H)$ over all graphs $H$ that can be obtained from $G$ by adding isolated vertices.
Note that $p_{2,\iso}(G)$ also satisfies the lower bounds \eqref{eq:vodd,maxdegree_lower} and \eqref{eq:la_lower}.
In particular, when $G$ has many (more than $2 \cdot \la(G)$) vertices of odd degree, then $p_{2,\iso}(G)$ can be arbitrarily large compared to $\la(G)$. On the other hand, we suspect that this is essentially the only way to get a large gap in the inequality $p_{2,\iso}(G)\geq \la(G)$; in fact, for Eulerian graphs $G$ ({\em i.e.}, graphs $G$ with $v_\odd(G) = 0$), we obtain the equality $p_{2,\iso}(G) = \la(G)$. 
\begin{restatable}{thm}{isolated}\label{thm:p2vsla}
Let $G$ be an Eulerian graph. Then
\[p_{2,\iso}(G) = \la(G).\]
 Moreover, if $\la(G) \leq 2$, then
\[p_2(G) = \la(G).\]
\end{restatable}
In the same section,
we exhibit a family of Eulerian graphs $G$ for which $p_{2,\iso}(G)<p_2(G)$. This implies, in particular, that $p_2(G)$ is not always the maximum of $\frac{v_{\odd}(G)}{2}, \left\lceil \frac{\Delta(G)}{2}\right \rceil$, and $a(G)$.

In Section~\ref{sec:upperbounds}, we prove a general upper bound for $p_2(G)$ resembling Equation~\eqref{eq:vodd,maxdegree_lower}. For path decompositions, $\max{ \left\{ \frac{v_\odd (G)}{2} , \frac{\Delta(G)}{2} \right\} }$ is often a poor lower bound on $p(G)$. For example, a vertex-disjoint union of $\ell$ cycles has $\max\left\{\frac{v_\odd}{2}, \lceil \frac{\Delta}{2} \rceil\right\} = 1$, while its path decomposition number is clearly $2\ell$. In contrast, the path odd-cover number of this graph is $2$ (see Figure~\ref{fig:cyclecover}).
It turns out that $\max{ \left\{ \frac{v_\odd (G)}{2} , \frac{\Delta(G)}{2} \right\} }$ describes $p_2(G)$ quite well in general. One of our main results is an upper bound on $p_2(G)$ in terms of $\Delta(G)$ and $v_{\odd}(G)$ of the following form.

\begin{restatable}{thm}{DeltaVSodd}
\label{thm:upperboundDeltaVSo2}
Let $G$ be a graph. Then 
\begin{align*}
    p_2(G) \leq \max\left\{\frac{v_{\odd}(G)}{2}, 2\left\lceil \frac{\Delta(G)}{2}\right \rceil\right\}.
\end{align*}
\end{restatable}
In particular, for any $n$-vertex graph $G$ with $\frac{v_\odd(G)}{2} \geq 2\left\lceil \frac{\Delta(G)}{2}\right\rceil$, the number of odd-degree vertices $v_{\odd}(G)$ determines the path odd-cover number exactly. On the other hand, if $G$ has few odd-degree vertices, then the bound on $p_2(G)$ given in Theorem \ref{thm:upperboundDeltaVSo2} has the advantage of being in terms of $\Delta(G)$ instead of $n$ as in the upper bounds on $p(G)$ from Lov\'asz \cite{lovasz1968covering}, Yan \cite{yan1998}, and Dean and Kouider \cite{dean2000}. This improvement is made possible by the use of nonedges of the graph in the odd-cover setting. We note that a similar phenomenon occurs in the context of bounding the combinatorial and the circuit diameter of partition polytopes \cite{b-13, bv-19a}. In Section~\ref{sec:firstupperbound} we elaborate on this connection and adapt the techniques in \cite{b-13} to derive a first, weaker upper bound on $p_2(G)$. In Section~\ref{sec:improvedupperbound} we refine these techniques to prove Theorem~\ref{thm:upperboundDeltaVSo2}.

In Section~\ref{sec:cycleOdd}, we adapt our arguments to cycle odd-covers of Eulerian graphs, and we obtain bounds analogous to those of Theorems~\ref{thm:topological}, \ref{thm:p2vsla}, and \ref{thm:upperboundDeltaVSo2}. Recall that the cycle odd-cover number $c_2(G)$ denotes the minimum cardinality of a cycle odd-cover of $G$. Note that $G$ admits a cycle odd-cover if and only if $G$ is Eulerian since the class of Eulerian graphs is closed under symmetric differences. This class is also closed under taking subdivisions and adding isolated vertices, so we may define $c_{2,\topo}(G)$ and $c_{2,\iso}(G)$ analogously to $p_{2,\topo}(G)$ and $p_{2,\iso}(G)$ for any Eulerian graph $G$, respectively. Specifically, we define $c_{2,\topo}(G)$ to be the minimum value of $c_2(H)$ over all subdivisions $H$ of $G$, and we define $c_{2,\iso}(G)$ to be the minimum value of $p_2(H)$ over all graphs $H$ that can be obtained from $G$ by adding isolated vertices.

An immediate lower bound for $c_2(G)$ is 
\[c_2(G) \geq \frac{\Delta(G)}{2}.\]
Taking $G$ to be $d$-regular with edge connectivity less than $d$ makes the inequality strict. The quantity $\Delta(G)/2$ is also a lower bound for $c_{2,\topo}(G)$ and $c_{2,\iso}(G)$. For upper bounds, the following bounds follow straightforwardly from our methods on path odd-covers.

\begin{restatable}{thm}{CycleOddCover}\label{thm:cycleOddCover}
    For any Eulerian graph $G$, we have the following.
    \begin{enumerate}
        \item $c_{2,\topo}(G) = \frac{\Delta(G)}{2}$, provided $G$ is not a union of two or more vertex-disjoint cycles;
        \item $c_{2,\iso}(G) \leq \la(G)$;
        \item $c_2(G) \leq \Delta(G)$.
    \end{enumerate}
\end{restatable}

In Section \ref{sec:openquestions}, we conclude with some open questions that arise from our work.

\section{Topological graphs and isolated vertices}\label{sec:lowerbounds}

We have seen that the path odd-cover number of a graph $G$ is bounded below by $v_{\odd}(G)/2$, $\la(G)$, and, in turn, $\lceil \Delta(G)/2 \rceil$ and $\max_{H \subseteq G} e(H) / (v(H) - 1)$ (see Equations~\eqref{eq:vodd,maxdegree_lower}, \eqref{eq:la_lower}, and~\eqref{eq:density_lower}).
In this section, we introduce two related invariants to show that the bound $\max\{\lceil \Delta(G)/2 \rceil, v_{\odd}(G)/2\}$ is tight for some subdivision of $G$ (with a few exceptional cases), and that if $G$ is Eulerian, then by adding isolated vertices to $G$ we can achieve the bound $\la(G)$.
Additionally, we present a family of Eulerian graphs for which adding an isolated vertex decreases the path odd-cover number; more precisely, for this family of graphs $p_2(G)$ exceeds the maximum of these bounds by $1$.

\subsection{Topological problem} 
\label{subsec:lowertopo}
Here we consider what types of constraints the topological structure of a graph can give on its path odd cover number. A \emph{topological graph} is an equivalence class of graphs, where $G_1 \sim G_2$ if $G_1$ and $G_2$ have a common subdivision, up to isomorphism. Note that if $H$ is a subdivision of $G$, then a path odd-cover of $G$ naturally gives a path odd-cover of $H$, and so we have $p_2(H) \leq p_2(G)$. We define the \emph{topological path odd-cover number} $p_{2,\topo}(G)$ to be the minimum of $p_2(H)$ over all subdivisions $H$ of $G$. We can think of this as the path odd-cover number of $G$ as a topological graph. By design, $p_{2,\topo}(G) \leq p_2(G)$.

The quantities $v_{\odd}(G)/2$ and $\lceil \Delta(G)/2 \rceil$ are invariants of a topological graph $G$, and so we have
\begin{equation}\label{eq:vodd,maxdegree_lower_topological}
    \max\left\{\frac{v_{\odd}(G)}{2}, \left\lceil \frac{\Delta(G)}{2}\right \rceil\right\} \leq p_{2,\topo}(G).
\end{equation}
If $G$ is the disjoint union of at least one cycle with at most one path, then $\max\left\{\frac{v_{\odd}(G)}{2}, \left\lceil \frac{\Delta(G)}{2}\right \rceil\right\} = 1$, while $p_{2,\topo}(G) = 2$. It turns out that this is the only case in which the inequality in (\ref{eq:vodd,maxdegree_lower_topological}) is strict. That is, $v_{\odd}(G)/2$ and $\lceil \Delta(G)/2 \rceil$ are essentially the only topological constraints on $p_2(G)$.
\Topological*

To prove this result, we first define a \emph{path $k$-system} to be a family $\{\mathcal P_1, \ldots, \mathcal P_k\}$, where each $\mathcal P_i$ is a non-empty collection of vertex-disjoint paths in $K_n$, such that each endpoint $v$ of a path in $\mathcal P := \bigcup_{i = 1}^k \mathcal P_i$ is one of the following types:
\begin{itemize}
    \item \emph{type I}: $v$ is the endpoint of exactly one path from $\mathcal P$, counting repetitions;
    \item \emph{type II}: $v$ is the endpoint of exactly two paths $P$ and $Q$ from $\mathcal P$ (from distinct collections $\mathcal P_i$ and $\mathcal P_j$),
    counting repetitions, and the respective terminal edges $uv \in P$ and $wv \in Q$ are distinct; further, $v$ appears in no other path of $\mathcal P$.
\end{itemize}
We call a path $k$-system \emph{well-distributed} if the following conditions hold:
\begin{enumerate}[(i)]
\item Each $\mathcal P_i$ has at most two type I endpoints.
\item If some $\mathcal P_i$ has two type I endpoints, then every $\mathcal P_i$ has at least one type I endpoint.
\item No single path $P \in \mathcal P$ has two type I endpoints, unless some $\mathcal P_i = \{P\}$.
\end{enumerate}
We now prove the following lemma, from which we will derive Theorem \ref{thm:topological}.
\begin{lem}\label{lem:path system}
Let $\{\mathcal P_1, \ldots, \mathcal P_k\}$ be a well-distributed path $k$-system with $\mathcal P = \bigcup_{i = 1}^k \mathcal P_i$, and let $G$ be a graph with edge set
\[E(G) = \bigoplus_{P \in \mathcal P}P.\]
Then
\[p_{2, \topo}(G) \leq k.\]
\end{lem}
\begin{proof}
    We induct on $|\mathcal P|$. If $|\mathcal P| = k$, then the result is trivial. We now assume $|\mathcal P| > k$.
    
    The goal is to replace our well-distributed path $k$-system with one which preserves $ \bigoplus_{P \in \mathcal P} P$ and reduces $|\mathcal P|$ by at least one. We use two operations to accomplish this.
    
    The first is applied to two type II endpoints $u$ and $v$ shared by the same two collections $\mathcal P_i$ and $\mathcal P_j$ in four distinct paths $P, P' \in \mathcal P_i$ and $Q, Q' \in \mathcal P_j$. We replace $\mathcal P_i$ with $\mathcal P_i \setminus \{P, P'\} \cup \{P \oplus uv \oplus P'\}$ and $\mathcal P_j$ with $\mathcal P_j \setminus \{Q, Q'\} \cup \{Q \oplus uv \oplus Q'\}$. In essence, this adds a double edge $uv$ to join two pairs of paths $P, P' \in \mathcal P_i$ and $Q, Q' \in \mathcal P_j$. We denote this operation by $\join(u,v)$.

    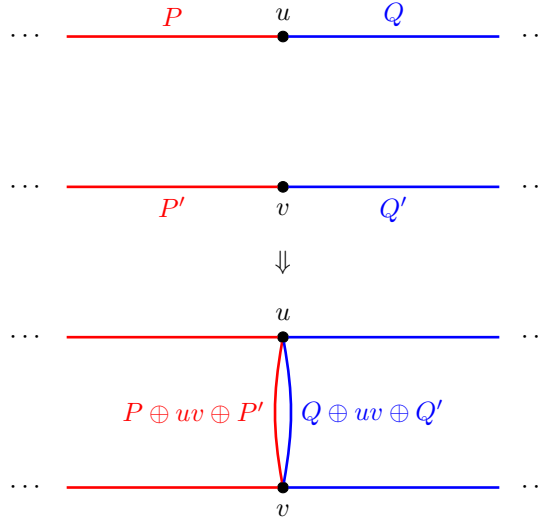
\begin{figure}[ht]
  \centering
  \begin{tikzpicture}[vert/.style={draw, fill=black, circle, inner sep=0pt, minimum size = 4pt, outer sep=0pt}, r_edge/.style={draw=red,line width= 1,>=latex,red}, b_edge/.style={draw=blue,line width= 1,>=latex,blue}, k_edge/.style={draw=black,line width= 1,>=latex,black}, label/.style = {draw=none, fill=none}]
    \node[vert] (u) at (0,2) {};
    \node (p1) at (-2,2) {};
    \node(p2) at (-3,2) {};
    \node (q1) at (2,2) {};
    \node(q2) at (3,2) {};

    \node[vert] (v) at (0,0) {};
    \node (pp1) at (-2,0) {};
    \node(pp2) at (-3,0) {};
    \node (qq1) at (2,0) {};
    \node(qq2) at (3,0) {};
    
    \draw (u) edge [r_edge] node[label, above] {$P$} (p2);
    \draw[b_edge] (u) edge node[label, above] {$Q$} (q2);
    
    \draw (v) edge [r_edge] node[label, below] {$P'$} (pp2);
    \draw (v) edge [b_edge] node[label, below] {$Q'$} (qq2);

    \node[label] (shift) at (0,-4) {};

    \node[label] at (0,-1) {$\Downarrow$};
    
    \node[vert] (ju) at ($(u) + (shift)$) {};
    \node (jp1) at ($(p1) + (shift)$) {};
    \node(jp2) at ($(p2) + (shift)$) {};
    \node (jq1) at ($(q1) + (shift)$) {};
    \node(jq2) at ($(q2) + (shift)$) {};

    \node[vert] (jv) at ($(v) + (shift)$) {};
    \node (jpp1) at ($(pp1) + (shift)$) {};
    \node(jpp2) at ($(pp2) + (shift)$) {};
    \node (jqq1) at ($(qq1) + (shift)$) {};
    \node(jqq2) at ($(qq2) + (shift)$) {};
    
    \draw (ju) edge [r_edge] (jp2);
    \draw[b_edge] (ju) edge (jq2);
    
    \draw (jv) edge [r_edge] (jpp2);
    \draw (jv) edge [b_edge] (jqq2);

    \draw (ju) edge[r_edge,bend right=10] node[label, left] {$P \oplus uv \oplus P'$} (jv);
    \draw (ju) edge[b_edge,bend left=10] node[label, right] {$Q \oplus uv \oplus Q'$} (jv);
    
    \node[label](hoffset) at (0.4,0) {};
    \node[label](voffset) at (0,-0.3) {};
    \node[label] at ($(u)-(voffset)$) {$u$};
    \node[label] at ($(v)+(voffset)$) {$v$};
    \node[label] at ($(p2)-(hoffset)$) {$\cdots$};
    \node[label] at ($(pp2)-(hoffset)$) {$\cdots$};
    \node[label] at ($(q2)+(hoffset)$) {$\cdots$};
    \node[label] at ($(qq2)+(hoffset)$) {$\cdots$};
    \node[label] at ($(ju)-(voffset)$) {$u$};
    \node[label] at ($(jv)+(voffset)$) {$v$};
    \node[label] at ($(jp2)-(hoffset)$) {$\cdots$};
    \node[label] at ($(jpp2)-(hoffset)$) {$\cdots$};
    \node[label] at ($(jq2)+(hoffset)$) {$\cdots$};
    \node[label] at ($(jqq2)+(hoffset)$) {$\cdots$};
  \end{tikzpicture}
  \caption{The operation $\join(u,v)$. \label{fig:join operation}}
\end{figure}

    The second is applied, assuming $k \geq 3$, to a type II endpoint $u$ shared by $P \in \mathcal P_i$ and $Q \in \mathcal P_j$ with $i \neq j$. We subdivide the terminal edge $zu$ of $P$ into $zu = zz' \oplus z'u$, add the one-edge path $z'u$ to some collection $\mathcal P_r$ distinct from $\mathcal P_i$ and $\mathcal P_j$, and replace $\mathcal P_i$ with $\mathcal P_i \setminus \{P\} \cup \{P \oplus z'u\}$.
    The effect is inserting a new $\mathcal P_r$ path at $u$ by subdividing the incident edge from $\mathcal P_i$. We denote this operation by $\ins(u, \mathcal P_i, \mathcal P_r)$.

\begin{figure}[ht]
  \centering
  \begin{tikzpicture}[vert/.style={draw, fill=black, circle, inner sep=0pt, minimum size = 4pt, outer sep=0pt}, r_edge/.style={draw=red,line width= 1,>=latex,red}, b_edge/.style={draw=blue,line width= 1,>=latex,blue}, y_edge/.style={draw=yellowish,line width= 1,>=latex,yellowish}, p_edge/.style={draw=orange,line width= 1,>=latex,orange}, label/.style = {draw=none, fill=none}]
    \node[vert] (u) at (0,2) {};
    \node[vert] (p1) at (-2,2) {};
    \node(p2) at (-3,2) {};
    \node[vert] (q1) at (2,2) {};
    \node(q2) at (3,2) {};

    \node[vert] (uu) at (0,0) {};
    \node[vert] (pp1) at (-2,0) {};
    \node(pp2) at (-3,0) {};
    \node[vert] (qq1) at (2,0) {};
    \node(qq2) at (3,0) {};
    \node[vert] (zz) at (-1,0) {};
    
    \draw (u) edge [r_edge] node[label, above] {$P$} (p1);
    \draw[r_edge] (p1) -- (p2);
    \draw[b_edge] (u) edge node[label, above] {$Q$} (q1);
    \draw[b_edge] (q1) -- (q2);
    
    \draw (zz) edge [r_edge] node[label, below] {$P \oplus z'u$} (pp1);
    \draw[r_edge] (pp1) -- (pp2);
    \draw (v) edge [b_edge] node[label, below] {$Q$} (qq1);
    \draw[b_edge] (qq1) -- (qq2);
    \draw (zz) edge [p_edge] node[label, above] {$z'u$} (uu);

    \node at (0,1) {$\Downarrow$};
    
    \node[label](hoffset) at (0.4,0) {};
    \node[label](voffset) at (0,-0.3) {};
    \node[label] at ($(u)-(voffset)$) {$u$};
    \node[label] at ($(p1)-(voffset)$) {$z$};
    \node[label] at ($(uu)-(voffset)$) {$u$};
    \node[label] at ($(zz)-(voffset)$) {$z'$};
    
    \node[label] at ($(p2)-(hoffset)$) {$\cdots$};
    \node[label] at ($(pp2)-(hoffset)$) {$\cdots$};
    \node[label] at ($(q2)+(hoffset)$) {$\cdots$};
    \node[label] at ($(qq2)+(hoffset)$) {$\cdots$};
  \end{tikzpicture}
  \caption{The operation $\ins(u,\mathcal P_i, \mathcal P_r$). \label{fig:insert operation}}
\end{figure}
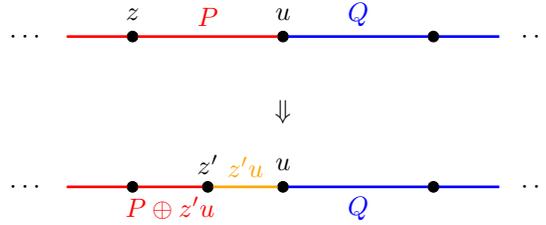
    
    We note that each of these operations indeed gives a new path $k$-system and preserves $\bigoplus_{P \in \mathcal P} P$ (up to subdivision of edges). Furthermore, they each preserve type I endpoints, so conditions (i) and (ii) of the definition of a well-distributed path $k$-system are inherited from the old system. Thus condition (iii) is the only one that needs to be checked. The operation $\join(u,v)$ can violate this condition only if $P$ and $P'$ each have a type I endpoint and $|\mathcal P_i| \geq 3$, or if the analogous situation holds for $\mathcal P_j$. The operation $\ins(u, \mathcal P_i, \mathcal P_r)$ can violate condition (iii) only if $\mathcal P_r$ has no type II endpoints; that is, if $\mathcal P_r = \{R\}$ for some path $R$ with two type I endpoints. We will ensure that these situations are avoided whenever applying $\join$ and $\ins$.
    
    We also note that $\join$ decreases the value of $|\mathcal P|$ by 2, while $\ins$ increases it by 1. Thus it suffices, for our inductive step, to apply $\join$ after at most one application of $\ins$.
    
    We assume without loss of generality that $|\mathcal P_1| = \max_i |\mathcal P_i|$, which implies $|\mathcal P_1| \geq 2$. We consider several cases.
    
    \begin{description}
        \item[Case $|\mathcal P_1| \geq 3$] Since $\mathcal P_1$ has at at most two type I endpoints, some $P \in \mathcal P_1$ has two type II endpoints $v$ and $v'$. Let $P'$ be another path in $\mathcal P_1$, which must have a type II endpoint $u$. Let $Q \in \mathcal P_i$ be the path in $\mathcal P$ meeting $P'$ at $u$, and let $R \in \mathcal P_j$ and $R' \in \mathcal P_{j'}$ be the respective paths in $\mathcal P$ meeting $P$ at $v$ and $v'$. Now if $j \neq i$, then we are free to use the operations $\ins(v, \mathcal P_j, \mathcal P_i)$ and $\join(u,v)$ in succession since $u$ is a type II vertex in $\mathcal P_i$ when we apply $\ins$; and $v'$ and the new subdividing vertex are both type II when we apply $\join$. We can thus assume that $j = i$, and by the same reasoning, that $j' = i$. Let $y$ be the other endpoint of $Q$, let $z$ be the other endpoint of $R$, and let $z'$ be the other endpoint of $R'$. Note that at most two of $y, z, z'$ can be type I. If not, then the paths $Q, R, R'$ would necessarily be distinct, and $y,z,z'$ would be distinct type I endpoints in $\mathcal P_i$, a contradiction.
        \begin{description}
            \item[Subcase $y$ type II] By relabeling vertices $v$ and $v'$ if necessary, we can assume that $Q \neq R$. Now we can apply the operation $\join(u,v)$ since $v'$ is type II in $\mathcal P_1$ and $y$ is type II in $\mathcal P_i$.
            \item[Subcase $y$ type I] By relabeling $v$ and $v'$ if necessary, we can assume that $z$ is type II. Now we can apply $\join(u,v)$ since $z$ is type II in $\mathcal P_i$ and $v'$ is type II in $\mathcal P_1$.
        \end{description}
         
        \item[Case $|\mathcal P_1| = 2$] Consider distinct $P, P' \in \mathcal P_1$. Let $u$ be a type II vertex in $P$, and let $v$ be a type II vertex in $P'$. Let $x$ be the other endpoint of $P$, and let $w$ be the other endpoint of $P'$. Let $Q, R \in \mathcal P$ be the other two paths with an endpoint at $u$ and $v$, respectively. Unless both $x$ and $w$ are type I, we can assume $Q \neq R$ by relabeling if necessary. For instance, if $x$ is type II and $Q=R$, then we could swap the roles of $x$ and $u$ and redefine $Q$ accordingly.
        \begin{description}
            \item[Subcase $Q \neq R$] If $Q$ and $R$ are in the same $\mathcal P_i$, then since $|\mathcal P_i| \leq |\mathcal P_1| = 2$, we can apply $\join(u,v)$ with no danger of violating condition (iii). If $
            Q \in \mathcal P_i$ and $R \in \mathcal P_j$ with $i \neq j$, then we can apply $\ins(v, \mathcal P_j, \mathcal P_i)$ and $\join(u,v)$ in succession.
            \item[Subcase $x, w$ type I, $Q = R$] Note that $\mathcal P_1$ has two type I endpoints, so by condition (ii), the collection $\mathcal P_i$ containing $Q$ has some type I endpoint. Both endpoints of $Q$, namely $u$ and $v$, are type II, so $\mathcal P_i$ has some other path $Q'$ with a type I endpoint $z$ and a type II endpoint $y$. Note that $y$ cannot be an endpoint of $\mathcal P_1$, so we can apply $\ins(y, \mathcal P_i, \mathcal P_1)$ to create a new type II endpoint $z'$ shared by $\mathcal P_1$ and $\mathcal P_i$, then apply $\join(u, z')$.
        \end{description}
    \end{description}
\end{proof}

\begin{proof}[Proof of Theorem \ref{thm:topological}]
    Let $k = \max\left\{\left\lceil \frac{\Delta(G)}{2} \right\rceil, \frac{v_{\odd}(G)}{2}\right\}$. We already know $p_{2, \topo}(G) \geq k$, so we must establish the other bound. If $k = 1$, then $\Delta(G) \leq 2$, so $G$ is a disjoint union of cycles and paths. Since $v_{\odd}(G) \leq 2$, $G$ has at most one component which is a path. Thus $G$ has no cycle components, by assumption. Clearly, $G$ has at least one edge, so $G$ is a path, and we have $p_2(G) = 1 = k$.
    
    Now suppose $k \geq 2$. We will produce a well-distributed path $k$-system $\{\mathcal P_1, \ldots, \mathcal P_k\}$ with $\bigoplus_{i = 1}^k \bigoplus_{P \in \mathcal P_i} P = E(H)$ for a subdivision $H$ of $G$, and the result will follow from Lemma \ref{lem:path system}. Subdivide each edge of $G$ into 3. Now $k$-color the edges in the following way. At each vertex $v$ of even degree at least 4, color the edges incident to $v$ so that each color is used 0 or 2 times. Order the odd-degree vertices. For $1 \leq i \leq k$, at the $i$th and $(k + i)$th odd-degree vertex (if they exist), color one incident edge with color $i$ and the remaining incident edges with the remaining colors so that each remaining color is used 0 or 2 times. Color the remaining edges so that each path between vertices of degree not 2 and each cyclic component uses at least two colors. The color classes now give a well-distributed path $k$-system.
\end{proof}

\subsection{Isolated vertices}
\label{subsec:loweriso}

In this section, we prove Theorem \ref{thm:p2vsla}, then we discuss its consequences for graphs which are not necessarily Eulerian. We then exhibit an example of a family of graphs $G$ for which adding an isolated vertex decreases the value of $p_2(G)$.

Recall from the introduction that $p_{2,\iso}(G)$ is the minimum value of $p_2(H)$ over all graphs $H$ that can be obtained from $G$ by adding isolated vertices.
Note that the established lower bounds $\la(G)$, $\frac{v_{\odd}(G)}2,$ and $\left \lceil \frac{\Delta(G)}{2}\right \rceil$ for $p_2(G)$ also hold for $p_{2,\iso}(G)$. We show that the linear arboricity bound is actually tight whenever $G$ is Eulerian.

\isolated*

\begin{proof}
    We will prove the second statement first.
    We have already observed that $\la(G) \leq p_2(G)$, so it suffices to show that $p_2(G) \leq 2$ whenever $\la(G) \leq 2$. In fact, since $G$ is Eulerian, we only need to consider the case when $\la(G) = 2$, say $E(G) = E(H_1) \cup E(H_2)$ for linear forests $H_1, H_2$.
    Let $\mathcal P_1 = \{P_1, \ldots, P_{s_1}\}$ and $\mathcal P_2 = \{Q_1, \ldots, Q_{s_2}\}$ be the respective collections of path components in $H_1$ and $H_2$. Since $G$ is Eulerian, every endpoint of a $P_i$ is the endpoint of some $Q_j$, so we have $s_1 = s_2 =: s$. Furthermore, all of the endpoints occur at vertices of degree 2 in $G$; that is, $\{\mathcal P_1, \mathcal P_2\}$ gives a path 2-system with no type I endpoints. Let $\mathcal P = \mathcal P_1 \cup \mathcal P_2$.
    
    The goal, similar to our proof of Lemma \ref{lem:path system}, is to replace the path 2-system $\{\mathcal P_1, \mathcal P_2\}$ with one that preserves $\bigoplus_{P \in \mathcal P} P = E(G)$, has no type I endpoints, and reduces the value of $|\mathcal P_1| = |\mathcal P_2|$ until we have $|\mathcal P_1| = |\mathcal P_2| = 1$.
    
    We can assume that $|\mathcal P_1| = |\mathcal P_2| \geq 2$. We pick any type II endpoint $v$ shared by the paths $P \in \mathcal P_1$ and $Q \in \mathcal P_2$. Let $u$ be the other endpoint of $P$, and let $w$ be the other endpoint of $Q$. Let $P'$ be a path in $\mathcal P_1$ distinct from $P$. Since $P, P'$ are vertex-disjoint and $P'$ has two endpoints, $P'$ has some endpoint $x$ distinct from $u,v,w$. Let $Q'$ be the path in $\mathcal P_2$ with an endpoint at $x$. The paths $P, P', Q, Q'$ are all distinct, so we can perform the operation $\join(v,x)$ and obtain the desired path 2-system.

    Now we prove the first statement. Suppose we are given a decomposition of $G$ into $k$ linear forests $H_1, \ldots, H_k$. This may not give a path $k$-system, as the path endpoints may occur at vertices of degree greater than 2. Thus we are not justified in using the $\join$ operation and must do something new. However, the fact that $G$ is Eulerian guarantees that every endpoint is shared by at least two linear forests, which is all we need if we allow for the addition of isolated vertices. Let $\mathcal P$ be the collection of all path components from the linear forests $H_1, \ldots, H_k$.

    Our new operation is applied to four distinct paths $P, P', Q, R \in \mathcal P$ with $P,Q$ sharing an endpoint $u$, with $P',R$ sharing an endpoint $v$, and with $P,P'$ being components of the same linear forest $H_i$. Let $H_{j_1}$ and $H_{j_2}$ be the respective linear forests with components $Q$ and $R$, which may or may not be distinct. We add an additional isolated vertex $w$, we add the edge $uw$ to $H_i$ and $H_{j_1}$, and we add the edge $vw$ to $H_i$ and $H_{j_2}$. The resulting graphs $H_i, H_{j_1}, H_{j_2}$ are still linear forests, and $\bigoplus_{j =1}^k E(H_j)$ is preserved, but now $H_i$ has one fewer component. We denote this operation by $\meet(u,v)$.

    If distinct paths $P_1, P_2$ belong to the same forest $H_i$, there exist (not necessarily distinct) paths $Q_1, R_1$ and $Q_2, R_2$ from other (not necessarily distinct) forests that meet $P_1$ and $P_2$ at their endpoints. Without loss of generality, $Q_1 \neq Q_2$. Let $u$ be the common endpoint of $P_1, Q_1$, and let $v$ be the common endpoint of $P_2, Q_2$. We can now apply $\meet(u,v)$. We iterate until all our linear forests are paths.
\end{proof}

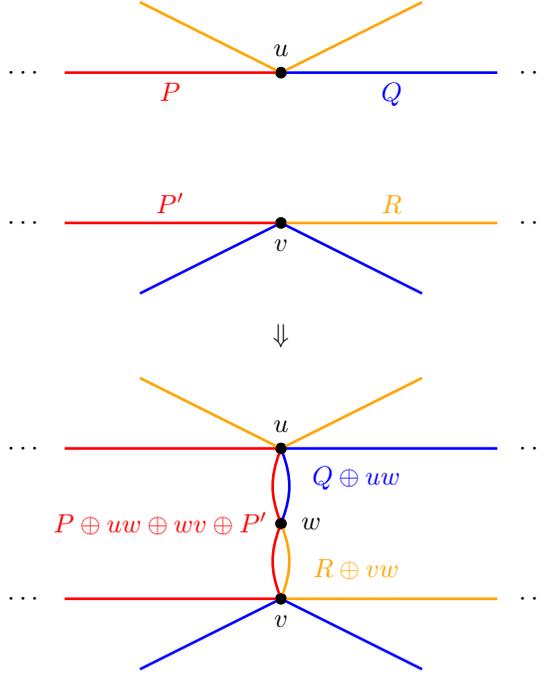
\begin{figure}[ht]
  \centering
  \begin{tikzpicture}[vert/.style={draw, fill=black, circle, inner sep=0pt, minimum size = 4pt, outer sep=0pt}, r_edge/.style={draw=red,line width= 1,>=latex,red}, b_edge/.style={draw=blue,line width= 1,>=latex,blue}, k_edge/.style={draw=black,line width= 1,>=latex,black}, y_edge/.style={draw=yellowish,line width= 1,>=latex,yellowish}, p_edge/.style={draw=orange,line width= 1,>=latex,orange}, label/.style = {draw=none, fill=none}]
    \node[vert] (u) at (0,2) {};
    \node (p1) at (-2,2) {};
    \node(p2) at (-3,2) {};
    \node (q1) at (2,2) {};
    \node(q2) at (3,2) {};
    \node (l1) at (-2,3) {};
    \node (r1) at (2,3) {};

    \node (c) at (0,1) {};

    \node[vert] (v) at (0,0) {};
    \node (pp1) at (-2,0) {};
    \node(pp2) at (-3,0) {};
    \node (qq1) at (2,0) {};
    \node(qq2) at (3,0) {};
    \node (ll1) at (-2,-1) {};
    \node (rr1) at (2,-1) {};
    
    \draw (u) edge [r_edge] node[label, below] {$P$} (p2);
    \draw[b_edge] (u) edge node[label, below] {$Q$} (q2);
    \draw[p_edge] (u) -- (l1);
    \draw[p_edge] (u) -- (r1);
    
    \draw (v) edge [r_edge] node[label, above] {$P'$} (pp2);
    \draw (v) edge [p_edge] node[label, above] {$R$} (qq2);
    \draw[b_edge] (v) -- (ll1);
    \draw[b_edge] (v) -- (rr1);

    \node[label] (shift) at (0,-5) {};

    \node[label] at (0,-1.5) {$\Downarrow$};
    
    \node[vert] (ju) at ($(u) + (shift)$) {};
    \node (jp1) at ($(p1) + (shift)$) {};
    \node(jp2) at ($(p2) + (shift)$) {};
    \node (jq1) at ($(q1) + (shift)$) {};
    \node(jq2) at ($(q2) + (shift)$) {};
    \node (jl1) at ($(l1) + (shift)$) {};
    \node (jr1) at ($(r1) + (shift)$) {};

    \node[vert] (jc) at ($(c) + (shift)$) {};

    \node[vert] (jv) at ($(v) + (shift)$) {};
    \node (jpp1) at ($(pp1) + (shift)$) {};
    \node(jpp2) at ($(pp2) + (shift)$) {};
    \node (jqq1) at ($(qq1) + (shift)$) {};
    \node(jqq2) at ($(qq2) + (shift)$) {};
    \node (jll1) at ($(ll1) + (shift)$) {};
    \node (jrr1) at ($(rr1) + (shift)$) {};
    
    \draw (ju) edge [r_edge] (jp2);
    \draw[b_edge] (ju) edge (jq2);
    \draw[p_edge] (ju) -- (jl1);
    \draw[p_edge] (ju) -- (jr1);
    
    \draw (jv) edge [r_edge] (jpp2);
    \draw[p_edge] (jv) -- (jqq2);
    \draw[b_edge] (jv) -- (jll1);
    \draw[b_edge] (jv) -- (jrr1);

    \draw (ju) edge[r_edge,bend right=20] (jc);
    \draw (jc) edge[r_edge,bend right=20] (jv);
    \draw (ju) edge[b_edge,bend left=20] (jc);
    \draw (jc) edge[p_edge,bend left=20] (jv);
    
    \node[label](hoffset) at (0.4,0) {};
    \node[label](voffset) at (0,-0.3) {};
    \node[label] at ($(u)-(voffset)$) {$u$};
    \node[label] at ($(v)+(voffset)$) {$v$};
    \node[label] at ($(p2)-(hoffset)$) {$\cdots$};
    \node[label] at ($(pp2)-(hoffset)$) {$\cdots$};
    \node[label] at ($(q2)+(hoffset)$) {$\cdots$};
    \node[label] at ($(qq2)+(hoffset)$) {$\cdots$};
    \node[label] at ($(ju)-(voffset)$) {$u$};
    \node[label] at ($(jv)+(voffset)$) {$v$};
    \node[label] at ($(jc)+(hoffset)$) {$w$};
    \node[label, red] at ($(jc)-4*(hoffset)$) {$P \oplus uw \oplus wv \oplus P'$};
    \node[label, blue] at ($(jc)+2.5*(hoffset)-2*(voffset)$) {$Q \oplus uw$};
    \node[label, orange] at ($(jc)+2.5*(hoffset)+2*(voffset)$) {$R \oplus vw$};
    \node[label] at ($(jp2)-(hoffset)$) {$\cdots$};
    \node[label] at ($(jpp2)-(hoffset)$) {$\cdots$};
    \node[label] at ($(jq2)+(hoffset)$) {$\cdots$};
    \node[label] at ($(jqq2)+(hoffset)$) {$\cdots$};
  \end{tikzpicture}
  \caption{The operation $\meet(u,v)$. \label{fig:meet operation}}
\end{figure}

For general graphs, we can obtain the following bound on $p_{2,\iso}(G)$.
\begin{restatable}{thm}{IsolatedGeneral}\label{thm:isolatedgeneral}
    Let $G$ be any graph, and let 
    \[t = \left\{\begin{array}{c l}
    2 & \text{if } v_\odd(G) = 4 \text{ and } \Delta(G) \geq 3, \\
    \left\lceil\frac{v_\odd(G)}{4}\right \rceil & \text{otherwise}
    \end{array}\right.\]
    and $d = 2\left \lceil \frac{\Delta(G)}{2} \right \rceil-2t$. Then
    \begin{align*}
        p_{2,\iso}(G) &\leq \left\{\begin{array}{l l}
        2t + \la(d) & \text{if } d \geq 0, \\
        \frac{v_{\odd}(G)}{2} & \text{if } d < 0,
        \end{array}\right.
    \end{align*}
    where $\la(d)$ denotes the maximum of $\la(H)$ over all graphs $H$ with maximum degree $d$.
\end{restatable}
We defer the proof of this theorem to the end of Section \ref{sec:upperbounds}, as the tools we use are the same as the ones we use later on for upper bounds on $p_2(G)$. However, we note here an immediate consequence to put this theorem in the context of our other bounds.
\begin{cor}\label{cor:isolatedasymptotic}
    For any graph $G$,
    \[p_{2,\iso}(G) \leq \max\left\{\left(\frac 12 + o(1)\right) \left(\Delta(G) + \frac{v_\odd(G)}{2}\right), \frac{v_\odd(G)}{2}\right\}\]
    as $\Delta(G) \to \infty$.
\end{cor}
\begin{proof}
    This comes from known bounds on linear arboricity from \cite{ferber2020towards}, namely
\[\la(d) = \frac{d}{2} + O(d^{2/3 - \alpha})\]
for some fixed $\alpha > 0$ as $d \to \infty$. If $p_{2,\iso}(G) < \frac{v_\odd(G)}{2}$, then in the notation of Theorem \ref{thm:isolatedgeneral}, we must have $d \geq 0$ and 
\begin{align*}
    p_{2,\iso}(G) &\leq 2t + \la(d) \leq 2\left \lceil \frac{v_\odd(G)}{4} \right \rceil + 2 + \frac{\Delta_e - 2\left \lceil\frac{v_\odd(G)}{4} \right \rceil}{2} + O(\Delta_e^{2/3 - \alpha}) \\
    &= \left(\frac{1}{2} + o(1)\right)\left(\Delta(G) + \frac{v_\odd(G)}{2}\right).
\end{align*}
\end{proof}
 
The proofs in this section rely heavily on the addition of isolated vertices in order to stitch together a linear forest decomposition into a path odd-cover. It is natural to wonder if the same thing can be done without these extra vertices, perhaps by other means; that is, can one obtain the same upper bounds for $p_2(G)$ that we obtained for $p_{2,\iso}(G)$? Unfortunately, this is not possible. We now show that adding isolated vertices can, in fact, reduce the path odd-cover number of a graph. The following proposition shows the separation of $p_2(G)$ and $p_{2,\iso}(G)$ in a family of Eulerian graphs.

\begin{prop}\label{prop:counteriso}
For every odd integer $k \geq 3$, there exists an Eulerian graph $G$ with $\la(G) = p_{2,\iso}(G) = k$ and $p_2(G) = k+1$.
\end{prop}
\begin{proof}
We begin by considering Walecki's classical cycle decomposition of $K_{2k+1}$~\cite{lucas1883recreations}. We take $V(K_{2k+1}) = \{v_j : j \in \mathbb Z_{2k} \cup \{\infty\}\}$. For each $i \in \{0, \ldots, k-1\}$, the $i$th cycle $C_i$ consists of edges of the form $v_jv_{2i-j}$ for $j \in \mathbb Z_{2k} \setminus \{i, i +k\}$, $v_jv_{2i+1-j}$ for $j \in \mathbb Z_{2k}$, as well as $v_iv_\infty$, and $v_{i+k}v_\infty$.

For each $i \in \{0, \ldots, k-1\}$, consider the path $P_i = C_i \setminus e_i$, where $e_i = v_0v_{2i+1} \in E(C_i)$ for $i \in \{0, \ldots, k-2\}$, and $e_{k-1} = v_1v_{2k-2} \in E(C_{k-1})$.
Let $H$ be the subgraph of $K_{2k+1}$ obtained by deleting the edges $e_0, \ldots, e_{k-1}$; that is, $E(H) = \bigoplus_{i = 0}^{k-1} P_i$.
Now $H$ has $k-1$ vertices of degree $2k-1$, namely $3, 5, \ldots, 2k-3,$ and $2k-2$, while the remaining vertices of $H$ have even degree. Furthermore, for each $i \in \{0, \ldots, k-1\}$, the endpoints of $P_i$ are the endpoints of $e_i$, of which at most one has odd degree. In fact, both endpoints of $P_0$ have even degree, and for $i \neq 0$, $P_i$ has exactly one odd-degree endpoint, which we denote by $v_{j_i}$.

We make a disjoint copy $H'$ of $H$ with an identical path decomposition $E(H') = \bigoplus_{i = 0}^{k-1} P_i'$. We denote the vertex in $H'$ corresponding to $v_j$ in $H$ by $v_j'$ for each $j \in \mathbb Z_{2k} \cup \{\infty\}$. We obtain our graph $G$ by adding the edges $v_{j_i}v_{j_i}'$ to $H \cup H'$ for $i \in \{1, \ldots, k-1\}$, which makes $G$ Eulerian. For $i \in \{1, \ldots, k-1\}$, let $Q_i = P_i \cup v_{j_i}v_{j_i'} \cup P_i'$. Now $E(G) = P_0 \oplus P_0' \oplus \bigoplus_{i = 1}^{k-1} Q_i$ is a path odd-cover (a path decomposition, in fact) of $G$ with $k+1$ paths, and $E(G) = (P_0 \cup P_0') \cup \bigcup_{i = 1}^{k-1} Q_i$ is a decomposition of $E(G)$ into $k$ linear forests.
Since $G$ is Eulerian, we have by Theorem \ref{thm:p2vsla} that
\[k = \frac{\Delta(G)}{2} \leq p_{2,\iso}(G) = \la(G) \leq k.\]

Now suppose by way of contradiction that $p_2(G) = k$, and consider a path odd-cover $E(G) = \bigoplus_{i = 1}^k R_i$ of $G$. Since $H$ and $H'$ each have vertices of degree $2k$, every path $R_i$ contains at least one edge with one endpoint in $H$ and one endpoint in $H'$. There are only $k-1$ such edges in $G$, so our odd-cover must cover some edge of $G$ at least 3 times or some nonedge of $G$ at least twice. Either way, we have
\[\sum_{i = 1}^k |E(R_i)| \geq |E(G)| + 2 = 2\left(\binom{2k+1}{2}-k\right)+(k-1) + 2 = 4k^2+k+1.\]
On the other hand, $G$ has only $2(2k+1)=4k+2$ vertices, so each $R_i$ has at most $4k+1$ edges, giving
\[\sum_{i = 1}^k |E(R_i)| \leq k(4k+1) = 4k^2 + k,\]
a contradiction. Hence $p_2(G) = k+1$.
\end{proof}

We show an example of this construction, for the case $k=3$, in Figure~\ref{fig:counteriso}. One thing to note about this construction is that it gives examples of graphs $G$ for which each of the lower bounds for $p_2(G)$ at the bottom of the Hasse diagram in Figure~\ref{fig:bound diagram} is not tight; that is, $p_2(G)$ is not simply the maximum of $\frac{v_\odd(G)}{2}$, $\left \lceil \frac{\Delta(G)}{2} \right \rceil$, and $a(G)$. There are many other graphs with this property, such as a disjoint union of $K_4$ with $K_3$, but we have yet to find a graph for which $p_2(G)$ exceeds this maximum by more than one.

\begin{figure}[t]
\centering

\begin{tikzpicture}[vert/.style={draw, fill=black, circle, inner sep=0pt, minimum size = 4pt, outer sep=0pt}, r_edge/.style={draw=red,line width= 1,>=latex,red}, b_edge/.style={draw=blue,line width= 1,>=latex,blue}, k_edge/.style={draw=gold,line width= 1,>=latex,gold}, scale=.7]

\foreach \a in {0,...,5}{
 \node[vert] (A\a) at (210-60*\a:3) {};
    \draw (210-60*\a:3.5) node{\a};
}
 \node[vert] (Ainf) at (0.7,0) {};
    \draw (0.9,-.5) node{$\infty$};
 
  \node[vert] (iso) at (4,4.5) {};

  \node (xshift) at (8cm,0) {};

 \foreach \a in {0,...,5}{
  \node[vert] (B\a) at ($(xshift) - (210-60*\a:3)$) {};
    \draw ($(xshift) - (210-60*\a:3.5)$) node{\a};
}
  \node[vert] (Binf) at ($(xshift) - (0.7,0)$) {};
    \draw ($(xshift) - (0.9,.5)$) node{$\infty$};

\draw [r_edge] (A3) -- (B3);
\draw [k_edge] (A4) -- (B4);

\draw [b_edge] (A1) to[bend left=30] (iso);
\draw [k_edge] (A1) to[bend left=40] (iso); 
\draw [b_edge] (B0) to[bend right=30] (iso);
\draw [r_edge] (B0) to[bend right=40] (iso); 

\draw [b_edge] (Ainf) -- (A3);
\draw [b_edge] (A3) -- (A4);
\draw [b_edge] (A2) -- (A4);
\draw [b_edge] (A2) -- (A5);
\draw [b_edge] (A1) -- (A5);
\draw [b_edge] (Ainf) -- (A0);

\draw [r_edge] (Ainf) -- (A1);
\draw [r_edge] (A2) -- (A1);
\draw [r_edge] (A2) -- (A0);
\draw [r_edge] (Ainf) -- (A4);
\draw [r_edge] (A5) -- (A4);
\draw [r_edge] (A5) -- (A3);

\draw [k_edge] (A4) -- (A0);
\draw [k_edge] (A5) -- (A0);
\draw [k_edge] (A5) -- (Ainf);
\draw [k_edge] (A2) -- (Ainf);
\draw [k_edge] (A2) -- (A3);
\draw [k_edge] (A1) -- (A3);

\draw [b_edge] (Binf) -- (B3);
\draw [b_edge] (B3) -- (B4);
\draw [b_edge] (B2) -- (B4);
\draw [b_edge] (B2) -- (B5);
\draw [b_edge] (B1) -- (B5);
\draw [b_edge] (Binf) -- (B0);

\draw [r_edge] (Binf) -- (B1);
\draw [r_edge] (B2) -- (B1);
\draw [r_edge] (B2) -- (B0);
\draw [r_edge] (Binf) -- (B4);
\draw [r_edge] (B5) -- (B4);
\draw [r_edge] (B5) -- (B3);

\draw [k_edge] (B4) -- (B0);
\draw [k_edge] (B5) -- (B0);
\draw [k_edge] (B5) -- (Binf);
\draw [k_edge] (B2) -- (Binf);
\draw [k_edge] (B2) -- (B3);
\draw [k_edge] (B1) -- (B3);

\end{tikzpicture}
\caption{An example of the construction $G+K_1$ in Proposition~\ref{prop:counteriso} when $k=3$. Here $p_2(G)=4$ and $p_{2,\iso}(G)=p_2(G+K_1)=3$.}\label{fig:counteriso}
\end{figure}
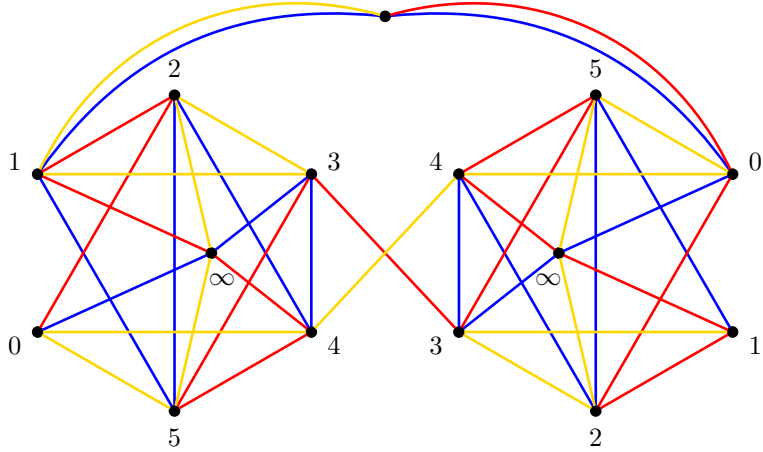

\section{Upper Bounds on $p_2(G)$}\label{sec:upperbounds}

In this section, we prove upper bounds on the path odd-cover number of a graph $G$ in terms of the maximum degree $\Delta$ (and $v_{\odd}$). Our discussion is presented in two stages. In Section \ref{sec:firstupperbound}, we devise a first upper bound through the adaptation of some tools from the literature.
The section is designed to lay out the general strategy of the proof.
First, the odd-degree vertices are connected through edge-disjoint paths,  whose deletion gives an Eulerian graph. Then it is possible to construct two paths whose symmetric difference gives a set of vertex-disjoint cycles that cover the vertices of maximal-degree. An iterative construction and deletion of such paths leads to a first bound of the form $\Delta+\frac{v_{\odd}}{2}$.

In Section \ref{sec:improvedupperbound}, we then refine the analysis by exhibiting that this construction works not only for a set of vertex-disjoint cycles, but also for the symmetric difference of such a set and two arbitrary edges. In doing so, we are able to perform the iterative construction and deletion routine immediately, without processing the odd-degree vertices separately first. This results in the bound in Theorem \ref{thm:upperboundDeltaVSo2} in which $v_{\odd}$ and $\Delta$ compete.

\subsection{A first upper bound in terms of $\Delta$}\label{sec:firstupperbound}

In this section, we prove the following first bound on the path odd-cover number.

\begin{thm}\label{thm:upperboundDelta+o2}
Let $G$ be a graph with maximum degree $\Delta$.
Then $E(G)$ can be odd-covered with at most $\Delta+\frac{v_{\odd}}{2}$ paths.
\end{thm}

Our approach is inspired by a strategy in \cite{b-13,bv-19a} where different partitions of the same data set are transformed into one another. The difference of two such partitions can be represented as a labeled, directed multigraph (called a {\em difference graph}) where the vertices correspond to the partition parts and each labeled edge corresponds to a data item that has to be moved. In a graph where all vertices have a balance of in- and out-degree, the vertices of maximal in-degree can be covered by a set of vertex-disjoint directed cycles. The item movements corresponding to such a set of cycles can be performed through application of two ``cylical exchanges'' (which correspond to labeled cycles between vertices of the graph). This leads to an iterative scheme in which one gradually reduces the maximal indegree of the remaining difference graph, and a bound in terms of said in-degree.

We are able to transfer some of this strategy to our setting, but have to be careful in doing so: some of the known concepts only hold specifically for bipartite graphs, labeled edges, directed networks, the existence of multiedges, and most importantly the use of cycles instead of paths.
Because of these challenges, the remainder of this section is dedicated to a self-contained explanation of the necessary tools tailored to our problem, culminating in the proof of Theorem \ref{thm:upperboundDelta+o2}. In Section \ref{sec:twoedges}, we will then be able to dramatically improve to the bound in Theorem \ref{thm:upperboundDeltaVSo2} -- where $\Delta$ and $\frac{v_{\odd}}{2}$ are not added but instead ``compete'' for the larger value -- through tools designed to better and more directly exploit structural results specific to the odd-cover setting.

We begin by proving that a set of vertex-disjoint cycles can be odd-covered by two paths.
If $\mathcal{C}=\{C_1,\dots,C_k\}$ is a set of vertex-disjoint cycles, then we also refer to the graph $\bigcup_{i=1}^k C_i$ as $\mathcal{C}$. In particular, we write $V(\mathcal{C})$ and $E(\mathcal{C})$ to denote $\bigcup_{i=1}^k V(C_i)$ and $ \bigcup_{i=1}^k E(C_i)$ respectively.

\begin{lemma} \label{lem:disjointcyclestwopaths}
Let $\mathcal{C}=\{C_1,\dots,C_k\}$ be a set of vertex-disjoint cycles.
Then $\mathcal{C}$ can be odd-covered using two paths $P$ and $Q$ with $V(P)\cup V(Q)\subseteq V(\mathcal{C})$.
\end{lemma}
\begin{proof}
For each $i\in[k]$, arbitrarily choose two distinct vertices $x_i,y_i \in V(C_i)$, and let $P_i,Q_i$ denote the two paths in $C_i$ from $x_i$ to $y_i$.
For each $i\in[k-1]$, let $e_i$ denote the edge $y_ix_{i+1}$.
Then the two paths 
\begin{align*}
    &P_1\cup e_1\cup P_2\cup e_2\cup\dots \cup P_{k-1}\cup e_{k-1}\cup P_k \text{ and} \\
    &
    Q_1\cup e_1\cup Q_2\cup e_2\cup\dots \cup Q_{k-1}\cup e_{k-1}\cup Q_k
\end{align*}
odd-cover $\mathcal{C}$. Figure \ref{fig:cyclecover} in the introduction depicts an example of this construction, where the paths $Q_i$ are all chosen as single-edge paths.
\end{proof}

This observation is intimately connected to some well-known constructions. Through a connection of $P_1$ to $P_k$ and $Q_1$ to $Q_k$ through the same edge $y_kx_1$, one obtains an odd-cover by a set of two cycles. This idea was the key ingredient for the famous result that Birkhoff polytopes have diameter $2$ \cite{br-74}. Equivalently, the concept can be represented as any permutation being the product of two indecomposable permutations.

Next, we show that, in an Eulerian graph $G$, the set of vertices of maximum degree can be covered by a set of vertex-disjoint cycles. This is known to hold for directed graphs that decompose into cycles \cite{b-13}; through a greedy orientation of the undirected cycles in $G$, it would be possible to transfer to our setting. We provide an elementary, alternative proof based on Hall's Theorem instead of network flows. Specifically, we use a correspondence between sets of vertex-disjoint cycles in a directed graph and perfect matchings in a certain associated bipartite graph. A similar correspondence appears for example in \cite{guenin2011packing}.

\begin{lemma} \label{lem:disjointcyclesmaxdegvertices}
Let $G$ be an Eulerian graph. Then there is a set of vertex-disjoint cycles of $G$ containing every vertex of maximum degree in $G$.
\end{lemma}
\begin{proof}
Since $G$ is Eulerian, there is a balanced orientation $D=(V,A)$ of $G$ with indegree being equal to outdegree at each vertex.
Let us construct an undirected bipartite graph $B$ associated with $D$ as follows.
For each $v\in V(D)$, we have two vertices $v^{out},v^{in}$ in $B$, and for each arc $uv \in A(D)$, we add the edge $u^{out}v^{in}$ to $B$.
Additionally, we add the edge $v^{out}v^{in}$ to $B$ for every vertex $v\in V(G)$ that is not a maximum degree vertex. Figure \ref{fig:bipartitematching} depicts an example of this construction. It is easy to see that there is a bijection between perfect matchings of $B$ and sets of vertex-disjoint cycles of $D$ containing every vertex of maximum degree in $G$.

\begin{figure}[t]
\centering
\begin{tikzpicture}[vertices/.style={draw, fill=black, circle, inner sep=0pt, minimum size = 4pt, outer sep=0pt}, scale=1.2]
\path[use as bounding box] (-1,-2) rectangle (9,5);
\node[vertices,label={[xshift=-0.2cm, yshift=-0.1cm]1}] (w_1) at (0,0) {};
\node[vertices,label={[xshift=-0.2cm, yshift=-0.1cm]2}] (w_2) at (0,1.5) {};
\node[vertices,label={[xshift=0.22cm, yshift=-0.5cm]4}] (w_3) at (1.5, 1.5) {};
\node[vertices,label={[xshift=0.25cm, yshift=-0.35cm]3}] (w_4) at (1.5, 0) {};
\node[vertices,label={[xshift=-0.2cm, yshift=-0.2cm]5}] (w_6) at (1.5, 3) {};
\node[vertices,label={[xshift=0.2cm, yshift=-0.1cm]6}] (w_7) at (3, 1.5) {};
\node[vertices,label={[xshift=0.2cm, yshift=-0.1cm]7}] (w_8) at (3, 3) {};

\foreach \to/\from in {w_1/w_2,w_2/w_6, w_6/w_7,
w_7/w_3, w_3/w_4, w_4/w_1}
\draw[draw=red, -stealth, line width= 1,>=latex]  (\to)--(\from);

\foreach \to/\from in {w_3/w_2,  
w_7/w_8,w_8/w_6,w_6/w_3,
w_2/w_4,w_4/w_7}
\draw[draw=black, -stealth, line width= 1,>=latex]  (\to)--(\from);

\node[vertices,label={[xshift=-0.4cm, yshift=-0.25cm]1$^{out}$}] (c1) at (6.0,-1.5) {};
\node[vertices,label={[xshift=-0.4cm, yshift=-0.25cm]2$^{out}$}] (c2) at (6.0,-0.5) {};
\node[vertices,label={[xshift=-0.4cm, yshift=-0.25cm]3$^{out}$}] (c3) at (6.0,0.5) {};
\node[vertices,label={[xshift=-0.4cm, yshift=-0.25cm]4$^{out}$}] (c4) at (6.0,1.5) {};
\node[vertices,label={[xshift=-0.4cm, yshift=-0.25cm]5$^{out}$}] (c5) at (6.0,2.5) {};
\node[vertices,label={[xshift=-0.4cm, yshift=-0.25cm]6$^{out}$}] (c6) at (6.0,3.5) {};
\node[vertices,label={[xshift=-0.4cm, yshift=-0.25cm]7$^{out}$}] (c7) at (6.0,4.5) {};

\node[vertices,label={[xshift=0.4cm, yshift=-0.25cm]1$^{in}$}] (d1) at (8.0,-1.5) {};
\node[vertices,label={[xshift=0.4cm, yshift=-0.25cm]2$^{in}$}] (d2) at (8.0,-0.5) {};
\node[vertices,label={[xshift=0.4cm, yshift=-0.25cm]3$^{in}$}] (d3) at (8.0,0.5) {};
\node[vertices,label={[xshift=0.4cm, yshift=-0.25cm]4$^{in}$}] (d4) at (8.0,1.5) {};
\node[vertices,label={[xshift=0.4cm, yshift=-0.25cm]5$^{in}$}] (d5) at (8.0,2.5) {};
\node[vertices,label={[xshift=0.4cm, yshift=-0.25cm]6$^{in}$}] (d6) at (8.0,3.5) {};
\node[vertices,label={[xshift=0.4cm, yshift=-0.25cm]7$^{in}$}] (d7) at (8.0,4.5) {};

\foreach \to/\from in {
c1/d2,c2/d5,
c5/d6,c6/d4,
c4/d3,c3/d1,c7/d7}
\draw[draw=red, line width=1,>=latex]  (\to)--(\from);

\foreach \to/\from in {
c1/d1,
c2/d3,c4/d2,
c3/d6,c5/d4,c7/d5,c6/d7}
\draw[draw=black, line width= 1,>=latex]  (\to)--(\from);

\end{tikzpicture}

\caption{A bipartite graph $B$ (right) constructed from an Eulerian directed graph $G$ (left). A perfect matching of $B$ corresponds to a set of vertex-disjoint cycles in $G$ that contain every vertex of maximum degree. Here, one cycle covers the vertices of maximum degree $2-6$. }\label{fig:bipartitematching}
\end{figure}
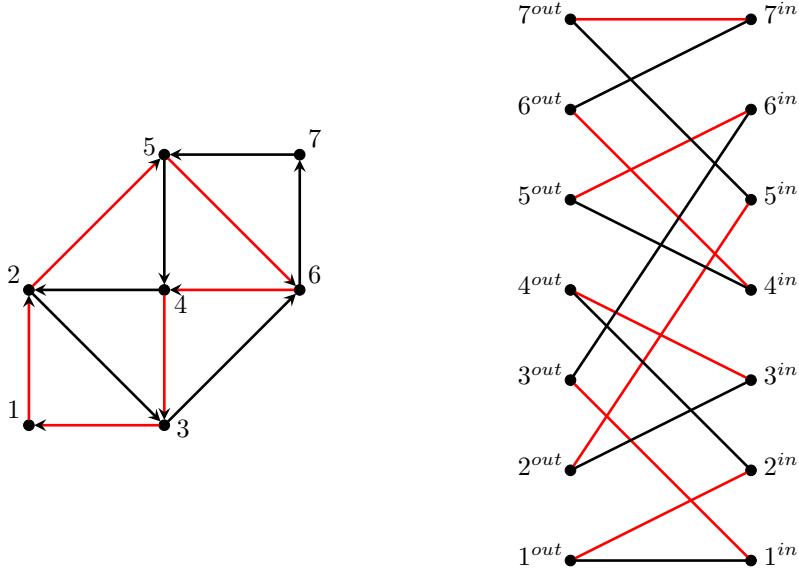

Let $V^{out} = \{v^{out}:v\in V(D)\}$ and $V^{in}=\{v^{in}:v\in V(D)\}$, so that $V(B) = V^{out}\cup V^{in}$.
For a vertex $u$ or a vertex set $U$ of $G$, we write $u^{out}$ or $U^{out}$ (respectively, $u^{in}$ or $U^{in}$) to denote the corresponding vertex or vertex set in $V^{out}$ (respectively, $V^{in})$, and vice versa.
We show that $B$ has a perfect matching by verifying Hall's condition; that is, we show that for all $U^{out}\subseteq V^{out}$, we have $|U^{out}|\leq |N_B(U^{out})|$.

Let $U^{out} \subseteq V^{out}$ and write $U^{out}=\{u_1^{out},u_2^{out},\dots,u_n^{out}\}$. We may assume without loss of generality that $\deg(u_1)\leq \dots \leq \deg(u_n)$.
Then there is a positive integer $k\leq n$ such that $u_i$ has maximum degree in $G$ if and only if $k<i\leq n$.
Consider the induced bipartite subgraph $B' = B[U^{out} \cup N_B(U^{out})]$.
We will count the edges of $B'$ from the two sides.
Counting on the side of $U^{out}$, we have
\begin{align*}
    |E(B')| &= \sum_{i=1}^n \deg_B(u_i^{out}) \\
    &= \sum_{i=1}^k \deg_B(u_i^{out}) + \sum_{i=k+1}^n \frac{\Delta}{2}.
\end{align*}
On the other hand, counting on the side of $N_B(U^{out})$, we have
\begin{align*}
    |E(B')| &= \sum_{i=1}^k \deg_{B'}(u_i^{in}) + \sum_{v \in N_B(U^{out})\setminus \{u_1^{in},\dots,u_k^{in}\}} \deg_{B'}(v).
\end{align*}
Since $\deg_{B'}(u_i^{in}) \leq \deg_B(u_i^{in})= \deg_B(u_i^{out})$ for all $i\in[k]$, we have
\begin{align} \label{eqn:sumdegB'(v)}
    \sum_{i=k+1}^n \frac{\Delta}{2} 
    &= \left(\sum_{i=1}^k (\deg_{B'}(u_i^{in}) - \deg_B(u_i^{out}))\right) +\sum_{v \in N_B(U^{out})\setminus \{u_1^{in},\dots,u_k^{in}\}} \deg_{B'}(v) \nonumber \\
    &\leq \sum_{v \in N_B(U^{out})\setminus \{u_1^{in},\dots,u_k^{in}\}} \deg_{B'}(v),
\end{align}
and since $\deg_{B'}(v)\leq \frac{\Delta}{2}$ for all $v\in V^{in}$, it follows that there are at least $(n-k)$ terms in the sum on the right hand side of \eqref{eqn:sumdegB'(v)}. In other words, we have $|N_B(U^{out}) \setminus \{u_1^{in},\dots,u_k^{in}\}| \geq n-k$, hence $|N_B(U^{out})|\geq |U^{out}|$ as desired.
\end{proof}

We are now ready to complete the proof of Theorem \ref{thm:upperboundDelta+o2}.

\begin{proof}[Proof of Theorem \ref{thm:upperboundDelta+o2}]
Let us first show that if $G$ is Eulerian and has maximum degree $\Delta$, then $G$ can be odd-covered using at most $\Delta$ paths.
We prove this by induction on $\Delta$. The base case $\Delta=2$ is given by Lemma \ref{lem:disjointcyclestwopaths}. 
Now suppose $\Delta\geq 4$. By Lemma \ref{lem:disjointcyclesmaxdegvertices}, there is a set $\mathcal{C}$ of vertex-disjoint cycles containing every maximum-degree vertex of $G$. Note that $E(\mathcal{C})$ can be odd-covered using 2 paths by Lemma \ref{lem:disjointcyclestwopaths}.
Now $G\setminus E(\mathcal{C})$ is clearly Eulerian, and moreover it has maximum degree $\Delta-2$ since every vertex of maximum degree in $G$ is incident with two edges in $E(\mathcal{C})$.  
By the inductive hypothesis, $G\setminus E(\mathcal{C})$ can be odd-covered using at most $\Delta-2$ paths, and it follows that $G$ can be odd-covered with at most $\Delta$ paths.

For the general case, note that there is a set of $\frac{v_{\odd}}{2}$ edge-disjoint paths in $G$ whose deletion leaves an Eulerian graph (this can be seen by iteratively deleting a path joining odd-degree vertices).
The remaining Eulerian graph can be odd-covered using at most $\Delta$ paths.
\end{proof}

\subsection{An improved bound where $\Delta(G)$ and $v_{\odd}$ compete}\label{sec:improvedupperbound}

Theorem \ref{thm:upperboundDelta+o2} establishes a first bound on the path odd-cover number in the form $\Delta+\frac{v_{\odd}}{2}$. We now refine our analysis to make the terms $\Delta$ and $\frac{v_{\odd}}{2}$ ``compete'' with each other. The result is the bound in Theorem~\ref{thm:upperboundDeltaVSo2}, restated below.

\DeltaVSodd*

This bound represents a dramatic improvement over Theorem \ref{thm:upperboundDelta+o2}. The main idea is the ``addition'' of two edges (incident to odd-degree vertices) to the set of vertex-disjoint cycles to be odd-covered by two paths in each iteration; {\em i.e.}, these edges are odd-covered along with the cycles.

Let us begin working towards a proof of Theorem \ref{thm:upperboundDeltaVSo2}. First, we transform $G=(V,E)$ into an Eulerian graph. Recall that the ability to cover the maximum-degree vertices by vertex-disjoint cycles is contingent on $G$ being Eulerian; see the proof of Lemma \ref{lem:disjointcyclesmaxdegvertices}.
To this end, we choose a perfect matching $M_{\odd}$ (in the underlying complete graph) between the odd-degree vertices of $G$ and replace the edge set $E$ by taking the symmetric difference with $M_{\odd}$, {\em i.e.}, $E'= E{\oplus}M_{\odd}$, so that  $G'=(V,E')$ is a simple Eulerian graph. The information on the edge set $M_{\odd}$ is stored separately.
Note that the transformation from $G$ to $G'$ increases the maximum degree from $\Delta$ to $\Delta_e:=2\left\lceil \frac{\Delta}{2} \right\rceil$ if $\Delta$ is odd and $M_{\odd}$ contains an edge incident to a maximum-degree vertex that is not in $E$.

The strategy now becomes to find a path odd-cover of $E = E'\oplus M_{\odd}$ by first applying Lemma \ref{lem:disjointcyclesmaxdegvertices} to $G'$ to obtain a set of vertex-disjoint cycles covering the maximum degree vertices of $G'$, then finding two paths that odd-cover the symmetric difference of this set of vertex-disjoint cycles and up to two additional edges from $M_\odd$.
We will refer to this as {\em integration} of edges in such a construction.  For readability and accessibility, we perform this discussion in multiple steps, gradually working from the easiest constructions towards the more involved, technical ones.

\subsubsection{Integration of one edge}\label{sec:oneedge}

We begin with the easier setting of adding one edge that is not in the edge set of the cycles. 
That is, given a set $\mathcal{C}=\{C_1,\dots,C_k\}$ of vertex-disjoint cycles and an edge $f\not\in E(\mathcal{C})$, we show that $\{f\}\cup E(\mathcal{C})$ can be odd-covered by two paths.
We will actually prove a slightly stronger statement. Note that the two endpoints of $f$ are the only vertices of odd degree in the edge set $\{f\}\cup E(\mathcal{C})$. 
Thus, if $P,Q$ are two paths odd-covering $\{f\}\cup E(\mathcal{C})$, then the endpoints of $f$ are both an endpoint of either $P$ or $Q$, and moreover $P$ and $Q$ share a common endpoint $z\in V(\mathcal{C})\setminus V(f)$. We show, in fact, that we can {\em choose} this common endpoint $z$ of the two paths odd-covering $\{f\}\cup E(\mathcal{C})$, with one exceptional case.

\begin{lemma}\label{lem:disjointcyclesplusedgeandcommonendpointtwopaths}
Let $\mathcal{C}=\{C_1,\dots,C_k\}$ be a set of vertex-disjoint cycles and let $f=uv$ be an edge not in $E(\mathcal{C})$.
Then the edge set $\{f\}\cup  E(\mathcal{C})$ can be odd-covered using two paths whose vertices are contained in $\{u,v\}\cup  V(\mathcal{C})$. Moreover, for every vertex $z\in V(\mathcal{C})\setminus \{u,v\}$, the two paths in the odd-cover can be chosen to have $z$ as a common endpoint, unless $k\geq 2$ and we have for some $j\in[k]$ that $u,v,z\in V(C_j)$.
\end{lemma}

\begin{proof}

Let us assume without loss of generality that $z\in V(C_k)$.
If $u$ and $v$ are both also in $V(C_k)$, then to prove the lemma we can assume $k=1$. In this case $\{f\}\cup E(\mathcal{C})$ consists of a cycle with a chord, which can clearly be decomposed into (hence odd-covered by) two paths.

So we may assume without loss of generality that $v\not\in V(C_k)$. 
If $v\in \bigcup_{i=1}^{k-1} V(C_i)$, then by relabelling the cycles $C_1,\dots,C_{k-1}$ we may also assume that $v\in V(C_1)$.
We choose distinct vertices $x_i,y_i\in V(C_i)$ for all $i\in[k]$ as follows; see Figure \ref{fig:cyclecommonendpoint}.
Let $y_k=z$. If $v\in V(C_1)$, then let $x_1=v$. Choose the remaining vertices $x_i,y_i \in V(C_i)$ arbitrarily so that none are equal to $u$; this is always possible because $|V(C_i)|\geq 3$ for all $i\in[k]$. 
Now label the two paths in $C_i$ from $x_i$ to $y_i$ as $P_i,Q_i$ so that $u\not\in V(P_i)$; this is always possible because $u\not\in\{x_i,y_i\}$.
Define $e_i=y_ix_{i+1}$ for $i\in[k-1]$.
Finally, if $v\in V(C_1)$ then let $R=\emptyset$, and otherwise let $R=vx_1$.
Then
\begin{align*}
    f\ \cup\ & R \cup P_1\cup e_1\cup P_2\cup e_2\cup\dots \cup P_{k-1}\cup e_{k-1}\cup P_k \text{ and} \\
    &
    R \cup Q_1\cup e_1\cup Q_2\cup e_2\cup\dots \cup Q_{k-1}\cup e_{k-1}\cup Q_k
\end{align*}
are both paths because $u\not\in V(P_i)$ for all $i\in[k]$. 
Moreover, the two paths odd-cover the edge set $\{f\}\cup  E(\mathcal{C})$, their vertex sets are contained in $\{u,v\}\cup  V(\mathcal{C})$, and they have a common endpoint $z$.
\end{proof}

Lemma \ref{lem:disjointcyclesplusedgeandcommonendpointtwopaths} can be applied to integrate a single edge $f \in M_{\odd}\backslash \mathcal{C}$, and the ability to choose $z$ will be a crucial tool in the following section.
For now, we observe that the proof also covers the situation where $f \in M_{\odd}\cap \mathcal{C}$.

\begin{lemma}\label{lem:disjointcyclesminusanyedgetwopaths}
Let $\mathcal{C}=\{C_1,\dots,C_k\}$ be a set of vertex-disjoint cycles and let $f\in  E(\mathcal{C})$.
Then the edge set $ E(\mathcal{C})\backslash \{f\}$ can be odd-covered using 2 paths.
\end{lemma}

\begin{proof}
Observe that the edge set $E(\mathcal{C})\backslash \{f\}$ consists of $k-1$ vertex-disjoint cycles and a path disjoint from those cycles. Let us assume without loss of generality that $f\in E(C_k)$. By Lemma \ref{lem:disjointcyclesplusedgeandcommonendpointtwopaths}, the edge set $\{f\} \cup E(\mathcal{C}\setminus\{C_k\})$ can be odd-covered using two paths whose vertices are contained in $V(f)\cup V(\mathcal{C} \setminus \{C_k\})$. Then exactly one of the two paths contains $f$, and by replacing $f$ in this path by the path $C_k\setminus \{f\}$, we obtain two paths odd-covering the edge set $E(\mathcal{C})\setminus \{f\}$.
\end{proof}

Lemmas \ref{lem:disjointcyclesplusedgeandcommonendpointtwopaths} and \ref{lem:disjointcyclesminusanyedgetwopaths} imply that any edge from $M_{\odd}$ can be integrated into any set of vertex-disjoint cycles. 

\begin{cor} \label{cor:cyclesoplusoneedge}
    Let $\mathcal{C}=\{C_1,\dots,C_k\}$ be a set of vertex-disjoint cycles and let $f$ be an edge. Then $\{f\}\oplus E(\mathcal{C})$ can be odd-covered using two paths.
\end{cor}

Thus, one can odd-cover one edge from $M_{\odd}$ at the same time two paths are used to reduce the maximum degree of the Eulerian graph $G$. By adjusting the iterative construction for the proof of Theorem \ref{thm:upperboundDelta+o2}, one arrives at an improved upper bound of $\max\{\Delta_e,\frac{v_{\odd}}{2}+\frac{1}{2}\Delta_e\}$ (recall that $\Delta_e = 2\left\lceil\frac{\Delta}{2}\right\rceil$).

Our goal is the much stronger bound of $\max\{\Delta_e,\frac{v_{\odd}}{2}\}$ claimed in Theorem \ref{thm:upperboundDeltaVSo2}. To this end, we have to integrate two edges from $M_{\odd}$, instead of just one, in each iteration.

\begin{figure}[t]
\centering
\begin{tikzpicture}[vertices/.style={draw, fill=black, circle, inner sep=0pt, minimum size = 4pt, outer sep=0pt}, r_edge/.style={draw=red,line width= 1,>=latex,red}, b_edge/.style={draw=blue,line width= 1,>=latex,blue}, k_edge/.style={draw=black,line width= 1,>=latex,black}, scale=1.8]
\node[vertices, label=below:$v\ $] (v) at (-1,0) {};
\node[vertices, label=below:$u$] (u) at (3,-0.35) {};

\node [vertices, label=below:$x_1\ \ $] (x_1) at (0,0) {};
\node [vertices, label=below:$\ \ \ y_1$] (y_1) at (1,0) {};
\node [inner sep=0pt, minimum size = 3pt, outer sep=0pt] (x_2) at (1.5,0) {};

\node [] at (1.75,0) {$\dots$};
\node [] at (4.25,0) {$\dots$};

\node [inner sep=0pt, minimum size = 3pt, outer sep=0pt] (y_i-1) at (2,0) {};
\node [vertices, label=below:$x_i\ \ \ $] (x_i) at (2.5,0) {};
\node [vertices, label=below:$\ \ \ y_i$] (y_i) at (3.5,0) {};
\node [inner sep=0pt, minimum size = 3pt, outer sep=0pt] (x_i+1) at (4,0) {};

\node [inner sep=0pt, minimum size = 3pt, outer sep=0pt] (y_k-1) at (4.5,0) {};
\node [vertices, label=below:$x_k\ \ \,$] (x_k) at (5,0) {};
\node [vertices, label=below:$\ \ \ y_k$] (y_k) at (6,0) {};

\node [vertices, label=above:$\ \ z$] (z) at (6,0) {};

\path [b_edge] (v) edge[bend left=30, looseness=0] (x_1);
\node [label=above:$R$] at (-0.5,0) {};
\path [b_edge] (x_1) edge[bend left=80] node[pos=0.5, above] {\color{black}$P_1$} (y_1);
\path [b_edge] (y_1) edge[bend left=30, looseness=0] (x_2);

\path [b_edge] (y_k-1) edge[bend left=30, looseness=0] (x_k);
\path [b_edge] (x_k) edge[bend left=80] node[pos=0.5, above] {\color{black}$P_k$} (y_k);

\path [r_edge] (v) edge[bend right=30, looseness=0] (x_1);
\path [r_edge] (x_1) edge[bend right=80] node[pos=0.5, below] {\color{black}$Q_1$} (y_1);
\path [r_edge] (y_1) edge[bend right=30, looseness=0] (x_2);

\path [r_edge] (y_k-1) edge[bend right=30, looseness=0] (x_k);
\path [r_edge] (x_k) edge[bend right=80]node[pos=0.5, below] {\color{black}$Q_k$}(y_k);

\path [b_edge] (u) edge[out=210,in=-70,looseness=0.6] node[pos=0.45, below] {\color{black}$f$}(v);

\path[b_edge] (y_i-1) edge[bend left=30, looseness=0] (x_i);
\path[r_edge] (y_i-1) edge[bend right=30, looseness=0] (x_i);

\path[r_edge] (x_i) edge[out=-90, in=180] (u);
\path[r_edge] (u) edge[out=0, in=-90] (y_i);
\path[b_edge] (x_i) edge[bend left=80] node[pos=0.5, above] {\color{black}$P_i$} (y_i);

\path[b_edge] (y_i) edge[bend left=30, looseness=0] (x_i+1);
\path[r_edge] (y_i) edge[bend right=30, looseness=0] (x_i+1);

\end{tikzpicture}

    \caption{One possible case in the proof of Lemma \ref{lem:disjointcyclesplusedgeandcommonendpointtwopaths} when $z\in V(C_k)$ and $v\not\in V(C_k)$. If $v\in V(\mathcal{C})$, then we set $x_1=v$ and $R=\emptyset$. The vertex $u$ could be anywhere, but we can always choose $x_i,y_i$ so that none are equal to $u$.}
    \label{fig:cyclecommonendpoint}
\end{figure}
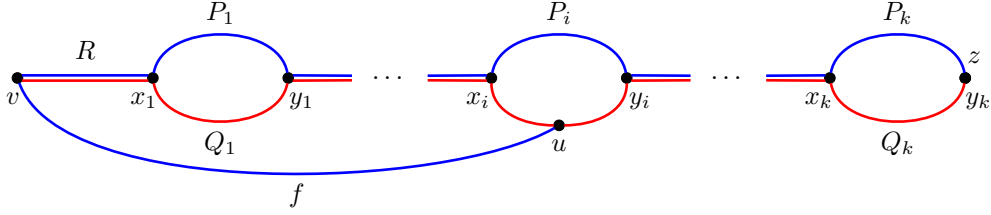

\subsubsection{Integration of two edges}\label{sec:twoedges}

We showed in Lemma \ref{lem:disjointcyclestwopaths} that the edges of any set $\mathcal{C}=\{C_1,\dots,C_k\}$ of vertex-disjoint cycles can be odd-covered using two paths. In the previous section, we extended this observation to the symmetric difference of the cycles and a single edge; c.f. Lemmas \ref{lem:disjointcyclesplusedgeandcommonendpointtwopaths} and \ref{lem:disjointcyclesminusanyedgetwopaths}. We begin this section by showing that if we are given two disjoint edges $f_1$ and $f_2$ ({\em i.e.}, they do not share a vertex) not in $ E(\mathcal{C})$, then two paths still suffice to odd-cover the edge set $\{f_1,f_2\}\cup  E(\mathcal{C})$, with one exceptional case.

If $C$ is a cycle and $v_1,\dots,v_k$, $k\geq 2$, are distinct vertices in $C$, then the edges of $C$ can uniquely be partitioned into $k$ paths, each with distinct endpoints in $\{v_1,\dots,v_k\}$. We call these paths the \emph{subpaths of $C$ demarcated by $v_1,\dots,v_k$}.
If $P$ is a path and $u,v \in V(P)$, then the subpath of $P$ from $u$ to $v$ is denoted $uPv$.

\begin{lemma}\label{lem:disjointcyclesplustwoedgestwopaths}

Let $\mathcal{C}=\{C_1,\dots,C_k\}$ be a set of vertex-disjoint cycles and let $f_1,f_2$ be disjoint edges not in $E(\mathcal{C})$.
Then $\{f_1,f_2\}\cup E(\mathcal{C})$ can be odd-covered using 2 paths, unless $k\geq 2$ and we have for some $j\in[k]$ that $\{f_1,f_2\}\cup C_j$ forms a subdivision of $K_4$ and at most one of the four subpaths of $C_j$ demarcated by the endpoints of $f_1,f_2$ has an internal vertex.
\end{lemma}
\begin{proof}
We proceed by a case split on which cycles of $\mathcal{C}$ contain the endpoints of $f_1,f_2$. Let $u_i,v_i$ denote the endpoints of $f_i$ for $i\in[2]$. 

\begin{description}
    \item[Case 1] There is a cycle of $\mathcal{C}$ containing all endpoints of $f_1,f_2$, say $u_1,v_1,u_2,v_2\in V(C_k)$.

    If $k=|\mathcal{C}|=1$, then $\{f_1,f_2\}\cup E(\mathcal{C})$ consists of a cycle with two disjoint chords, and it is easy to see that the edges of such a graph can be decomposed into (hence odd-covered by) two paths.
    So let us assume that $k\geq 2$. Let $P,Q$ be two paths odd-covering $\bigcup_{i=1}^{k-1}E(C_i)$ with $V(P)\cup V(Q) \subseteq \bigcup_{i=1}^{k-1}V(C_i)$ given by Lemma \ref{lem:disjointcyclestwopaths}, and let $x,y$ denote their common endpoints.

    We split into two further cases depending on whether $\{f_1,f_2\}\cup E(C_k)$ forms a $K_4$-subdivision or not. If it does not form a $K_4$-subdivision, then we may assume by symmetry that the four vertices $u_1,v_1,v_2,u_2$ occur in this cyclic order on $C_k$; otherwise, we may assume that they occur in the cyclic order $u_1,u_2,v_1,v_2$ on $C_k$.

    \begin{description}
        \item[Case 1.1] The endpoints of $f_1,f_2$ occur in the cyclic order $u_1,v_1,v_2,u_2$ on $C_k$.  See Figure \ref{fig:case:uvvu}.

        Let $R_1, R_2, R_u, R_v$ denote the four subpaths of $C_k$ demarcated by $u_1, v_1, v_2, u_2$ with endpoints $\{u_1,v_1\}$, $\{u_2,v_2\}$, $\{u_1,u_2\}$, $\{v_1,v_2\}$ respectively.
        Since $f_1=u_1v_1$ and $f_2=u_2v_2$ are edges not in $E(\mathcal{C})$, both $R_1$ and $R_2$ have at least one internal vertex, say $z_1$ and $z_2$ respectively.
        Then the two paths
        \begin{align*}
            u_1v_1 \cup v_1R_1z_1 \cup z_1x\ \cup\ &P \cup yz_2 \cup z_2R_2u_2 \cup u_2v_2 \text{ and} \\
            R_u \cup u_1R_1z_1 \cup z_1x\ \cup\ &Q \cup yz_2 \cup z_2R_2v_2\cup R_v
        \end{align*}
        odd-cover $\{f_1,f_2\}\cup  E(\mathcal{C})$.

\begin{figure}[t]
\centering
\begin{subfigure}[b]{0.5\textwidth}
\begin{tikzpicture}[vertices/.style={draw, fill=black, circle, inner sep=0pt, minimum size = 5pt, outer sep=0pt}, scale=1]
  \def\Radius{1.5}
  \path
    (0,0) coordinate (M)
    (\Radius, 0) coordinate (A)
    (M) + (45:\Radius) coordinate (B)
    (M) +(90:\Radius) coordinate (C) 
    (M) +(135:\Radius) coordinate (D)
    (M) +(180:\Radius) coordinate (E)
    (M) +(225:\Radius) coordinate (F)
    (M) +(270:\Radius) coordinate (G)
    (M) +(315:\Radius) coordinate (H)
  ;

  \draw[line width=1, red] (A) arc(0:45:\Radius);
  \draw[line width=1, blue] (B) arc(45:90:\Radius);
  \draw[line width=1, red] (C) arc(90:135:\Radius);
  \draw[line width=1, red] (D) arc(135:180:\Radius);
  \draw[line width=1, red] (E) arc(180:225:\Radius);
  \draw[line width=1, blue] (F) arc(225:270:\Radius);
  \draw[line width=1, red] (G) arc(270:315:\Radius);
  \draw[line width=1, red] (H) arc(315:360:\Radius);
  
\node at ([shift={(0.9,0.1)}]C) {$R_1$};
\node at ([shift={(-0.9,-0.1)}]G) {$R_2$};
  
\node[vertices, label=right:$u_1$] (A) at (A) {};
\node[vertices, label={[xshift=-0.15cm, yshift=-0.55cm]:$z_1$}] (B) at (B) {};
\node[vertices, label=above:$v_1$] (C) at (C) {};
\node[label={[xshift=-0.3cm, yshift=-0.1cm]$R_v$}] (D) at (D) {};
\node[vertices, label=left:$v_2$] (E) at (E) {};
\node[vertices, label={[xshift=0.15cm, yshift=-0.05cm]:$z_2$}] (F) at (F) {};
\node[vertices, label=below:$u_2$] (G) at (G) {};
\node[label={[xshift=0.4cm, yshift=-0.6cm]$R_u$}] (H) at (H) {};

\path[line width=1, blue] (A) edge[bend left=30] node[below left, xshift=0.15cm, yshift=0.1cm] {$\color{black} f_1$} (C);
\path[line width=1, blue] (E) edge[bend left=30] node[above right, xshift=-0.05cm, yshift=-0.1cm] {$\color{black} f_2$} (G);

\node[vertices, label=below:$x$] (x) at ($(B)+(1.5,0)$) {};
\node[vertices, label=below:$y$] (y) at ($(F)-(1.5,0)$) {};
\path [line width=1, blue] (B) edge[bend left=30, looseness=0] (x);
\path [line width=1, red] (B) edge[bend right=30, looseness=0] (x);
\path [line width=1, blue] (y) edge[bend left=30, looseness=0] (F);
\path [line width=1, red] (y) edge[bend right=30, looseness=0] (F);

\draw[dashed] (x) edge[out=130,in=110,looseness=1.8] node[pos=0.6, above left] {$P,Q$} (y);
\end{tikzpicture}
\label{fig:case:uvvu1}
\end{subfigure}

\caption{Case 1.1 in the proof of Lemma \ref{lem:disjointcyclesplustwoedgestwopaths}, where $u_1,v_1,v_2,u_2$ occur in this order on $V(C_k)$ and $k\geq 2$. Here $R_1$ is the path from $u_1$ to $v_1$ containing $z_1$, and $R_2$ is the path from $u_2$ to $v_2$ containing $z_2$.} 
\label{fig:case:uvvu}
\end{figure}
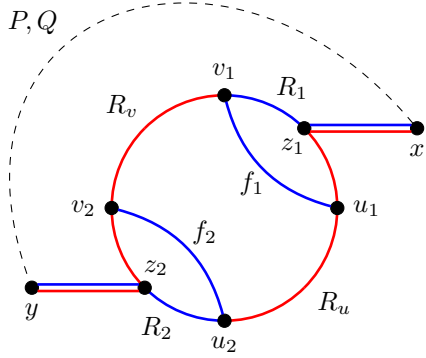

        \item[Case 1.2] The endpoints of $f_1,f_2$ occur in the cyclic order $u_1,u_2,v_1,v_2$ on $C_k$. See Figure \ref{fig:case:uuvv}.
                
        Let $R_u,R_v,R_1,R_2$ denote the four subpaths of $C_k$ demarcated by $u_1,u_2,v_1,v_2$ with endpoints $\{u_1,u_2\}, \{v_1,v_2\}, \{u_1,v_2\}$ and $\{u_2,v_1\}$ respectively.
        To prove the lemma we may assume that at least two of $R_u,R_v,R_1,R_2$ have internal vertices.
        There are two cases up to symmetry; two of $R_u,R_v,R_1,R_2$ having internal vertices are either disjoint or incident ({\em i.e.}, they share an endpoint). In the first case, we may assume by symmetry that $R_u,R_v$ have internal vertices. In the second case, we may assume that $R_u,R_1$ have internal vertices.

        If $R_u,R_v$ have internal vertices, say $z_u,z_v$ respectively, then the two paths
        \begin{align*}
            R_1 \cup v_2u_2 \cup u_2R_uz_u \cup z_ux\ \cup\ &P \cup yz_v\cup  z_vR_vv_1 \text{ and} \\
            R_2 \cup v_1u_1 \cup u_1R_uz_u \cup z_ux\ \cup\ &Q\cup yz_v \cup z_vR_vv_2
        \end{align*}
        odd-cover $\{f_1,f_2\}\cup E(\mathcal{C})$. See Figure \ref{fig:case:uuvv1}.
        
        If $R_u,R_1$ have internal vertices, say $z_u,z_1$ respectively, then the two paths 
        \begin{align*}
            v_1u_1 \cup u_1R_uz_u \cup z_ux\ \cup\ &P \cup yz_1 \cup z_1R_1v_2\cup v_2u_2 \text{ and} \\
            R_v \cup R_2 \cup u_2R_uz_u \cup z_ux\ \cup\ &Q \cup yz_1\cup z_1R_1u_1
        \end{align*}
        odd-cover $\{f_1,f_2\}\cup E(\mathcal{C})$. See Figure \ref{fig:case:uuvv2}. We have thus proved the lemma when the endpoints of $f_1,f_2$ all belong to one cycle of $\mathcal{C}$.
        Therefore, we may assume that $u_1,v_1,u_2,v_2$ do not all belong to one cycle.

        \begin{figure}[t]
        \centering
        \begin{subfigure}[b]{0.47\linewidth}
        \begin{tikzpicture}[vertices/.style={draw, fill=black, circle, inner sep=0pt, minimum size = 5pt, outer sep=0pt}, scale=1]
          \def\Radius{1.5}
          \path
            (0,0) coordinate (M)
            (\Radius, 0) coordinate (A)
            (M) + (45:\Radius) coordinate (B)
            (M) +(90:\Radius) coordinate (C) 
            (M) +(135:\Radius) coordinate (D)
            (M) +(180:\Radius) coordinate (E)
            (M) +(225:\Radius) coordinate (F)
            (M) +(270:\Radius) coordinate (G)
            (M) +(315:\Radius) coordinate (H)
          ;

          \draw[line width=1, red] (A) arc(0:45:\Radius);
          \draw[line width=1, blue] (B) arc(45:90:\Radius);
          \draw[line width=1, red] (C) arc(90:135:\Radius);
          \draw[line width=1, red] (D) arc(135:180:\Radius);
          \draw[line width=1, blue] (E) arc(180:225:\Radius);
          \draw[line width=1, red] (F) arc(225:270:\Radius);
          \draw[line width=1, blue] (G) arc(270:315:\Radius);
          \draw[line width=1, blue] (H) arc(315:360:\Radius);

        \node at ([shift={(0.9,0.1)}]C) {$R_u$};
        \node at ([shift={(-0.9,-0.1)}]G) {$R_v$};

        \node[vertices, label=right:$u_1$] (A) at (A) {};
        \node[vertices, label={[xshift=-0.15cm, yshift=-0.55cm]:$z_u$}] (B) at (B) {};
        \node[vertices, label=above:$u_2$] (C) at (C) {};
        \node[label={[xshift=-0.3cm, yshift=-0.1cm]$R_2$}] (D) at (D) {};
        \node[vertices, label=left:$v_1$] (E) at (E) {};
        \node[vertices, label={[xshift=0.15cm, yshift=-0.05cm]:$z_v$}] (F) at (F) {};
        \node[vertices, label=below:$v_2$] (G) at (G) {};
        \node[label={[xshift=0.4cm, yshift=-0.6cm]$R_1$}] (H) at (H) {};
        
        \path[line width=1, red] (A) edge node[pos=0.25,below,yshift=0.05cm] {$\color{black} f_1$} (E);
        \path[line width=1, blue] (C) edge node[pos=0.22, left,xshift=0.07cm] {$\color{black} f_2$} (G);
        
        \node[vertices, label=below:$x$] (x) at ($(B)+(1.5,0)$) {};
        \node[vertices, label=below:$y$] (y) at ($(F)-(1.5,0)$) {};
        \path [line width=1, blue] (B) edge[bend left=30, looseness=0] (x);
        \path [line width=1, red] (B) edge[bend right=30, looseness=0] (x);
        \path [line width=1, blue] (y) edge[bend left=30, looseness=0] (F);
        \path [line width=1, red] (y) edge[bend right=30, looseness=0] (F);
        
        \draw[dashed] (x) edge[out=130,in=110,looseness=1.8] node[pos=0.6, above left] {$P,Q$} (y);
        \end{tikzpicture}
        \caption{$R_u$ is the path from $u_1$ to $u_2$ containing $z_u$, and $R_v$ is the path from $v_1$ to $v_2$ containing $z_v$.}
        \label{fig:case:uuvv1}
        \end{subfigure}
        \hfill
        \begin{subfigure}[b]{0.47\linewidth}
        \begin{tikzpicture}[vertices/.style={draw, fill=black, circle, inner sep=0pt, minimum size = 5pt, outer sep=0pt}, scale=1]
          \def\Radius{1.5}
          \path
            (0,0) coordinate (M)
            (\Radius, 0) coordinate (A)
            (M) + (45:\Radius) coordinate (B)
            (M) +(90:\Radius) coordinate (C) 
            (M) +(135:\Radius) coordinate (D)
            (M) +(180:\Radius) coordinate (E)
            (M) +(225:\Radius) coordinate (F)
            (M) +(270:\Radius) coordinate (G)
            (M) +(315:\Radius) coordinate (H)
          ;
          
          \draw[line width=1, blue] (A) arc(0:45:\Radius);
          \draw[line width=1, red] (B) arc(45:90:\Radius);
          \draw[line width=1, red] (C) arc(90:135:\Radius);
          \draw[line width=1, red] (D) arc(135:180:\Radius);
          \draw[line width=1, red] (E) arc(180:225:\Radius);
          \draw[line width=1, red] (F) arc(225:270:\Radius);
          \draw[line width=1, blue] (G) arc(270:315:\Radius);
          \draw[line width=1, red] (H) arc(315:360:\Radius);
        
        \node at ([shift={(0.9,0.1)}]C) {$R_u$};
        \node at ([shift={(0.9,-0.1)}]G) {$R_1$};
          
        \node[vertices, label={[xshift=0.33cm,yshift=-0.32cm]:$u_1$}] (A) at (A) {};
        \node[vertices, label={[xshift=-0.15cm, yshift=-0.55cm]:$z_u$}] (B) at (B) {};
        \node[vertices, label=above:$u_2$] (C) at (C) {};
        \node[label={[xshift=-0.3cm, yshift=-0.1cm]$R_2$}] (D) at (D) {};
        \node[vertices, label=left:$v_1$] (E) at (E) {};
        \node[label={[xshift=-0.3cm, yshift=-0.5cm]$R_v$}] (F) at (F) {};
        \node[vertices, label=below:$v_2$] (G) at (G) {};
        \node[vertices, label={[xshift=-0.15cm, yshift=-0.05cm]:$z_1$}] (H) at (H) {};
        
        \path[line width=1, blue] (A) edge node[pos=0.75,below,yshift=0.05cm] {$\color{black} f_1$} (E);
        \path[line width=1, blue] (C) edge node[pos=0.22, left,xshift=0.07cm] {$\color{black} f_2$} (G);
        
        \node[vertices, label=below:$x$] (x) at ($(B)+(1.5,0)$) {};
        \node[vertices, label=below:$y$] (y) at ($(H)+(1.5,0)$) {};
        \path [line width=1, blue] (B) edge[bend left=30, looseness=0] (x);
        \path [line width=1, red] (B) edge[bend right=30, looseness=0] (x);
        \path [line width=1, red] (y) edge[bend left=30, looseness=0] (H);
        \path [line width=1, blue] (y) edge[bend right=30, looseness=0] (H);
        
        \draw[dashed] (x) edge[out=-30,in=30,looseness=1.2] node[pos=0.5, left] {$P,Q$} (y);
        \end{tikzpicture}
        \caption{$R_u$ is the path from $u_1$ to $u_2$ containing $z_u$, and $R_1$ is the path from $u_1$ to $v_2$ containing $z_1$.}
        \label{fig:case:uuvv2}
        \end{subfigure}
        
        \caption{Case 1.2 in the proof of Lemma \ref{lem:disjointcyclesplustwoedgestwopaths} where $u_1,u_2,v_1,v_2$ occur in this order on $V(C_k)$ and $k\geq 2$. This case is divided in to two subcases depending on which of the paths $R_1,R_2,R_u,R_v$ have at least one internal vertex.} 
        \label{fig:case:uuvv}
        \end{figure}
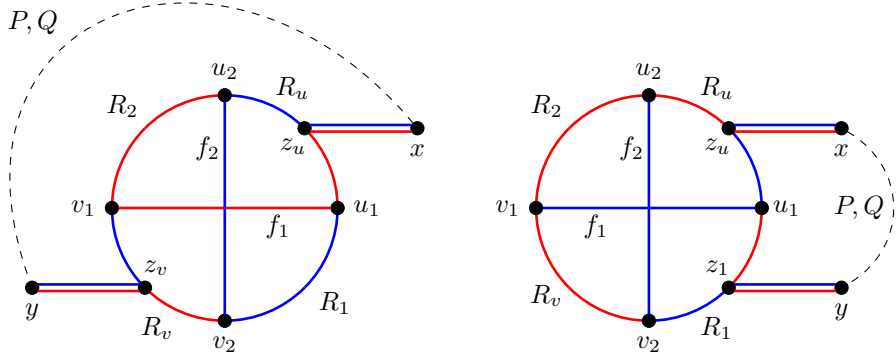
    \end{description}

    \item[Case 2] There is a cycle of $\mathcal{C}$ containing exactly three of the endpoints of $f_1,f_2$, say $u_1,u_2,v_2\in V(C_1)$ and $v_1\not\in V(C_1)$. See Figure \ref{fig:case:uuvinC1}.

    Let $R_2,R_u,R_v$ denote the three subpaths of $C_1$ demarcated by $u_1,u_2,v_2$ with endpoints $\{u_2,v_2\}, \{u_1,u_2\}, \{u_1,v_2\}$ respectively. 
    Since $f_2=u_2v_2$ is not in $E(\mathcal{C})$, $R_2$ has an internal vertex, say $y_1$.
    For all $i\in[k]\setminus \{1\}$, choose arbitrary $x_i,y_i\in V(C_i)$ and $P_i,Q_i$ as before, except that if $v_1\in V(C_i)$ for some $i\neq 1$, then we reorder $C_2,\dots,C_{k}$ and choose $y_k=v_1$. In this case also define $R=\emptyset$. If $v_1\not\in  V(\mathcal{C})$, then define $R=y_kv_1$. For $i\in[k-1]$ let $e_i=y_ix_{i+1}$. 
    Then the two paths
    \begin{align*}
        u_2v_2\ \cup v_2R_2y_1 \cup e_1\cup P_2\cup e_2\ \cup &\dots \cup P_{k-1}\cup e_{k-1}\cup P_k \cup R \cup v_1u_1 \text{ and} \\
        R_v\cup R_u\cup u_2R_2y_1 \cup e_1\cup Q_2\cup e_2\ \cup &\dots \cup Q_{k-1}\cup e_{k-1}\cup Q_k \cup R
    \end{align*}
    odd-cover $\{f_1,f_2\}\cup  E(\mathcal{C})$.

    \begin{figure}[t]
    \centering
    \begin{tikzpicture}[vertices/.style={draw, fill=black, circle, inner sep=0pt, minimum size = 4pt, outer sep=0pt}, r_edge/.style={draw=red,line width= 1,>=latex,red}, b_edge/.style={draw=blue,line width= 1,>=latex,blue}, k_edge/.style={draw=black,line width= 1,>=latex,black}, scale=1.8]
    \node[vertices, label=below:$u_1\ \ $] (u_2) at (0,0) {};
    \node[vertices, label=below:$u_2$] (v_2) at (0.5,-0.35) {};
    \node[vertices, label=above:$v_2$] (u_1) at (0.5,0.35) {};
    
    \node [vertices, label=below:$\ \ \ y_1$] (y_1) at (1,0) {};
    \node [vertices, label=below:$x_2$] (x_2) at (2,0) {};
    \node [] at (2.5,0) {$\dots$};
    \node [vertices, label=below:$y_{k-1}$] (y_k-1) at (3,0) {};
    \node [vertices, label=below:$x_k\ \ $] (x_k) at (4,0) {};
    \node [vertices, label=below:$\ \ \ y_k$] (y_k) at (5,0) {};
    \node[vertices, label=below:$v_1$] (v_1) at (6,0) {};
    \node [label=below:$R$] at (5.5,0) {};
    
    \path [r_edge] (v_2) edge[out=180,in=-80] (u_2);
    \path [r_edge] (u_2) edge[out=80,in=180] (u_1);
    \path [b_edge] (u_1) edge[out=0,in=100] (y_1);
    \path [r_edge] (v_2) edge[out=0, in=-100] (y_1);
    \path [b_edge] (v_2) edge node[pos=0.5, left,xshift=0.07cm] {$\color{black} f_2$} (u_1);
    
    \path [b_edge] (y_1) edge[bend left=30, looseness=00,looseness=0] (x_2);
    \path [r_edge] (y_1) edge[bend right=30, looseness=00,looseness=0] (x_2);
    \path [b_edge] (y_k-1) edge[bend left=30, looseness=00,looseness=0] (x_k);
    \path [r_edge] (y_k-1) edge[bend right=30, looseness=00,looseness=0] (x_k);
    \path [b_edge] (x_k) edge[bend left=80] node[pos=0.5,above]{\color{black}$P_k$} (y_k);
    \path [r_edge] (x_k) edge[bend right=80] node[pos=0.5,below]{\color{black}$Q_k$}(y_k);
    \path [b_edge] (y_k) edge[bend left=30, looseness=0] (v_1);
    \path [r_edge] (y_k) edge[bend right=30, looseness=0] (v_1);
    \path [b_edge] (u_2) edge[out=140, in=130, looseness=0.7] node [pos=0.5,below]{\color{black}$f_1$} (v_1);
    
    \end{tikzpicture}
    \caption{Case 2 in the proof of Lemma \ref{lem:disjointcyclesplustwoedgestwopaths} where $u_1,u_2,v_2\in V(C_1)$ and $v_1\not\in V(C_1)$. This figure depicts the case where $v_1\not\in \bigcup_{i=2}^k V(C_i)$. If $v_1\in \bigcup_{i=2}^k V(C_i)$, then $v_1=y_k$ and $R$ is empty. } 
    \label{fig:case:uuvinC1}
    \end{figure}
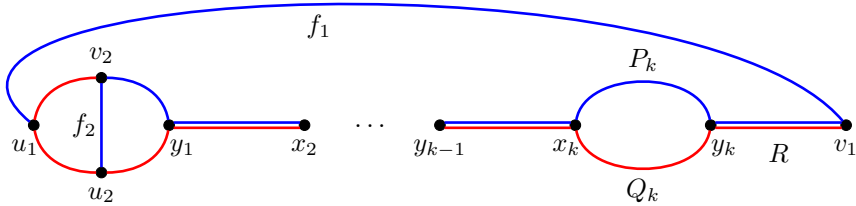

    \item[Case 3] Every cycle of $\mathcal{C}$ contains at most two of the endpoints of $f_1,f_2$.

    Let us first show that every endpoint of $f_1,f_2$ is in some cycle of $\mathcal{C}$.
    Suppose for instance that $u_1\not\in V(\mathcal{C})$. 
    Then we apply Lemma \ref{lem:disjointcyclesplusedgeandcommonendpointtwopaths} to obtain two paths $P,Q$ odd-covering $\{f_2\}\cup E(\mathcal{C})$ with a common endpoint $z$ chosen as follows: if $v_1\in  V(\mathcal{C})$, then let $z=v_1$ and define $R=\emptyset$ (note that the exceptional case of Lemma \ref{lem:disjointcyclesplusedgeandcommonendpointtwopaths} does not occur because of the assumption that every cycle of $\mathcal{C}$ contains at most two of the endpoints of $f_1,f_2$); otherwise, if $v_1\not\in  V(\mathcal{C})$, then let $z\not\in \{u_2,v_2\}$ be an arbitrary vertex in $V(\mathcal{C})$ that does not form an exceptional case, and let $R=zv_1$. Then the two paths $P\cup R \cup \{f_1\}$ and $Q\cup R$ odd-cover $\{f_1,f_2\} \cup  E(\mathcal{C})$.

    So we may assume that $\{u_1,v_1,u_2,v_2\}\subseteq V(\mathcal{C})$. We now consider two further cases: either there are two cycles of $\mathcal{C}$ each containing exactly two of the endpoints of $f_1,f_2$, or there is at most one cycle of $\mathcal{C}$ that contains two of the endpoints of $f_1,f_2$.

    \begin{description}

    \item[Case 3.1] At most one cycle of $\mathcal{C}$ contains two of $u_1,v_1,u_2,v_2$.

    Since $\{u_1,v_1,u_2,v_2\}\subseteq  V(\mathcal{C})$, this implies that there is a cycle in $\mathcal{C}$ containing exactly one of $u_1,v_1,u_2,v_2$; say $V(C_k)\cap \{u_1,u_2,v_1,v_2\}=\{u_1\}$. 
    See Figure \ref{fig:reroutecommonedge}.
    Since every cycle of $\mathcal{C}$ has at most two of $u_1,v_1,u_2,v_2$, it follows that $k\geq 3$.
    By Lemma \ref{lem:disjointcyclesplusedgeandcommonendpointtwopaths}, 
    there exist two paths $P,Q$ odd-covering $\{f_2\}\cup E(\mathcal{C}\setminus \{C_k\})$ 
    with $V(P)\cup V(Q)\subseteq V(\mathcal{C}\setminus \{C_k\})$ 
    and with $v_1$ as a common endpoint (again, the exceptional case of Lemma \ref{lem:disjointcyclesplusedgeandcommonendpointtwopaths} does not occur due to the assumption that every cycle of $\mathcal{C}$ contains at most two of the endpoints of $f_1,f_2$). Note that $u_2,v_2$ are the only vertices of odd degree in $\{f_2\}\cup E(\mathcal{C}\setminus\{C_k\})$, so they are each one endpoint of $P$ and $Q$. Let us assume without loss of generality that $u_2$ is an endpoint of $P$ and $v_2$ is and endpoint of $Q$.
    Since $\mathcal{C}\setminus \{C_k\}$ has at least two disjoint cycles (because $k\geq 3$), there is at least one edge $e=xy$ that belongs to both $P$ and $Q$. 
    We will reroute this common edge $e$ through $C_k$ as follows.
    Let $x_k,y_k$ be distinct vertices in $V(C_k)\setminus \{u_1\}$, and let $P_k,Q_k$ denote the two paths from $x_k$ to $y_k$ in $C_k$ with $u_1\not\in V(P_k)$.
    Then the two paths
    \begin{align*}
        P' &:= (P\setminus \{e\}) \cup xx_k \cup P_k \cup y_ky \\
        Q' &:= (Q\setminus \{e\}) \cup xx_k \cup Q_k \cup y_ky
    \end{align*}
    odd-cover $\{f_2\} \cup E(\mathcal{C})$. Moreover, since $u_1 \not\in V(P_k)$, $u_1v_1\cup P'$ is a path, therefore $u_1v_1\cup P'$ and $Q'$ are two paths odd-covering $\{f_1,f_2\}\cup E(\mathcal{C})$.

    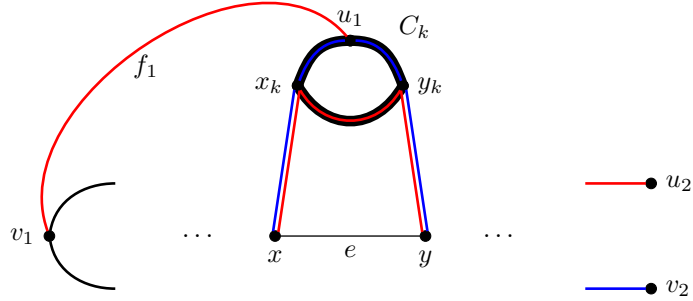
\begin{figure}
    \centering
    \begin{tikzpicture}[vertices/.style={draw, fill=black, circle, inner sep=0pt, minimum size = 4pt, outer sep=0pt}, r_edge/.style={draw=red,line width= 1,>=latex,red}, b_edge/.style={draw=blue,line width= 1,>=latex,blue}, k_edge/.style={draw=black,line width= 1,>=latex,black}, scale=2]
    \node[vertices, label=left:$v_1$] (v1) at (0,0) {};
    \node[] (v11) at (0.5,-0.35) {};
    \node[] (v12) at (0.5,0.35) {};
    
    \node [] at (1,0) {$\dots$};
    \node [vertices, label=below:$x$] (x) at (1.5,0) {};
    \node [vertices, label=below:$y$] (y) at (2.5,0) {};
    \path [draw, line width=0.5] (x) edge node [below] {$e$} (y);
    
    \node [] at (3,0) {$\dots$};
    \node[] (u') at (3.5,0.35) {};
    \node[] (v') at (3.5,-0.35) {};
    \node[vertices, label=right:$u_2$] (u2) at (4,0.35) {};
    \node[vertices, label=right:$v_2$] (v2) at (4,-0.35) {};
    \path[r_edge] (u') edge (u2);
    \path[b_edge] (v') edge (v2);
    
    \path [draw=blue,line width= 1] (v11) edge[out=180,in=-80] (v1);
    \path [draw=red,line width= 1] (v1) edge[out=80,in=180] (v12);
    
    \node[vertices, label=above:$u_1$] (u1) at (2,1.3) {};
    \node[vertices, label=left:$x_k$] (x_k) at (1.65,1) {};
    \node[vertices, label=right:$y_k$] (y_k) at (2.35,1) {};
    
    \path[draw=black!30, line width=4] (x_k.center) edge[bend right=60, looseness=1.33] (y_k.center);
    \path[draw=black!30, line width=4] (x_k.center) edge[out=70, in=180, looseness=1.2] (u1.center);
    \path[draw=black!30, line width=4] (u1.center) edge[out=0, in=110, looseness=1.2] node[pos=0.5, above right] {$C_k$}  (y_k.center);
    
    \node[vertices] at (2,1.3) {};
    \node[vertices] at (1.65,1) {};
    \node[vertices] at (2.35,1) {};
    
    \path[r_edge] (x) edge [bend right=30, looseness=0] (x_k);
    \path[b_edge] (x) edge [bend left=30, looseness=0] (x_k);
    \path[r_edge] (y_k) edge [bend right=30, looseness=0] (y);
    \path[b_edge] (y_k) edge [bend left=30, looseness=0] (y);

    \path[r_edge] (x_k) edge[bend right=60, looseness=1.2] (y_k);
    \path[b_edge] (x_k) edge[out=70, in=180] (u1);
    \path[b_edge] (u1) edge[out=0, in=110] (y_k);
    
    \path[r_edge] (v1) edge[out=110, in=130, looseness=1] node[pos=0.5, below] {\color{black}$f_1$} (u1);
    
    \end{tikzpicture}
    \caption{Case 3.1 in the proof of Lemma \ref{lem:disjointcyclesplustwoedgestwopaths} where $C_k$ contains $u_1$ and none of $\{u_2,v_1,v_2\}$. We find two paths that odd-cover $\{f_2\}\cup \bigcup_{i=1}^{k-1}C_i$, and reroute a common edge $e$ to odd-cover $\{f_1\}\cup C_k$ as well. } 
    \label{fig:reroutecommonedge}
    \end{figure}
    
        \item[Case 3.2] Two cycles of $\mathcal{C}$, say $C_1,C_2$, each contain exactly two of the endpoints of $f_1,f_2$.

        If $u_1,u_2\in V(C_1)$ and $v_1,v_2\in V(C_2)$, then $\{f_1,f_2\} \cup C_1\cup C_2$ can be alternatively decomposed similarly to Case 1.1 as follows. For $i\in[2]$ consider the two subpaths $P_i,Q_i$ of $C_i$ demarcated by the two vertices of $u_1,u_2,v_1,v_2$ in $V(C_i)$. Since $|V(C_i)|\geq 3$, we may assume without loss of generality that $P_i$ has an internal vertex. Let $C$ denote the cycle $\{f_1,f_2\}\cup P_1\cup P_2$ and let $f_1'=u_1u_2$ and $f_2'=v_1v_2$. It follows from Case 1.1 (as in Figure \ref{fig:case:uvvu}) that $\{f_1',f_2'\}\cup  E((\mathcal{C}\setminus\{C_1,C_2\})\cup \{C\})$
        can be odd-covered using two paths. By replacing the edges $f_1',f_2'$ with the paths $Q_1,Q_2$ respectively, we obtain an odd-cover of $\{f_1,f_2\}\cup E(\mathcal{C})$ using two paths.
        
        So we may assume that $u_1,v_1\in V(C_1)$ and $u_2,v_2\in V(C_2)$. See Figure \ref{fig:case:uvC1uvC2}.  
        Let $z_i \in V(C_i)\setminus \{u_i,v_i\}$ for $i\in[2]$. 
        Then by Lemma \ref{lem:disjointcyclesplusedgeandcommonendpointtwopaths}, for each $i\in[2]$, $\{f_i\}\cup E(C_i)$ can be odd-covered by two paths $P_i,Q_i$ with a common endpoint $z_i$. If $k\geq 3$, then letting $P,Q$ be two paths odd-covering $\bigcup_{i=3}^k E(C_i)$ with common endpoints $x,y$, we have that the two paths 
        \begin{align*}
            &P_1 \cup z_1x\cup P \cup yz_2\cup P_2 \qquad \text{ and} \\
            &Q_1 \cup z_1x\cup Q \cup yz_2 \cup Q_2
        \end{align*}
        odd-cover $\{f_1,f_2\}\cup  E(\mathcal{C})$.
        Otherwise, if $k=2$, then $P_1\cup z_1z_2\cup P_2$ and $Q_1\cup z_1z_2\cup Q_2$ odd-cover $\{f_1,f_2\}\cup  E(\mathcal{C})$.

        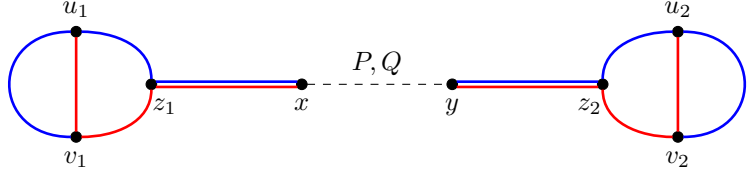
\begin{figure}[t]
        \centering
        \begin{tikzpicture}[vertices/.style={draw, fill=black, circle, inner sep=0pt, minimum size = 4pt, outer sep=0pt}, r_edge/.style={draw=red,line width= 1,>=latex,red}, b_edge/.style={draw=blue,line width= 1,>=latex,blue}, k_edge/.style={draw=black,line width= 1,>=latex,black}, scale=2]
        \node[vertices, label=below:$v_1$] (v_1) at (0.5,-0.35) {};
        \node[vertices, label=above:$u_1$] (u_1) at (0.5,0.35) {};
        
        \node [vertices, label=below:$\ \ \ z_1$] (z_1) at (1,0) {};
        \node [vertices, label=below:$x$] (x) at (2,0) {};
        \node [vertices, label=below:$y$] (y) at (3,0) {};
        \node [vertices, label=below:$z_2\ \ \ $] (z_2) at (4,0) {};
        \node[vertices, label=below:$v_2$] (v_2) at (4.5,-0.35) {};
        \node[vertices, label=above:$u_2$] (u_2) at (4.5,0.35) {};
        
        \draw[dashed] (x) edge node[pos=0.5,above]{$P,Q$} (y);
        
        \path [b_edge] (z_1) edge[out=90,in=0] (u_1);
        \path [b_edge] (u_1) edge[out=180,in=180, looseness=2] (v_1);
        \path [draw=blue,line width= 1,>=latex,red] (v_1) edge[out=0,in=-90] (z_1);
        \path [draw=blue,line width= 1,>=latex,red] (u_1) edge (v_1);
        
        \path [b_edge] (z_2) edge[out=90,in=180] (u_2);
        \path [b_edge] (u_2) edge[out=0,in=0, looseness=2] (v_2);
        \path [draw=blue,line width= 1,>=latex,red] (v_2) edge[out=180,in=-90] (z_2);
        \path [draw=blue,line width= 1,>=latex,red] (u_2) edge (v_2);
        
        \path [b_edge] (z_1) edge[bend left=30, looseness=00,looseness=0] (x);
        \path [r_edge] (z_1) edge[bend right=30, looseness=00,looseness=0] (x);
        \path [b_edge] (y) edge[bend left=30, looseness=00,looseness=0] (z_2);
        \path [r_edge] (y) edge[bend right=30, looseness=00,looseness=0] (z_2);
        
        \end{tikzpicture}
        \caption{Case 3.2 in the proof of Lemma \ref{lem:disjointcyclesplustwoedgestwopaths} where $u_1,v_1\in V(C_1)$ and $u_2,v_2\in V(C_2)$. If $k=2$, then the path between the cycles $C_1,C_2$ is replaced by the edge $z_1z_2$.} 
        \label{fig:case:uvC1uvC2}
        \end{figure}
    
    \end{description}
    
\end{description}
This completes the proof of Lemma \ref{lem:disjointcyclesplustwoedgestwopaths}.
\end{proof}

Next, we show that if both $f_1$ and $f_2$ lie in the edge set $ E(\mathcal{C})$, then the set of edges of the cycles without $f_1$ and $f_2$ can be odd-covered by two paths.

\begin{lemma}\label{lem:disjointcyclesminustwoedgestwopaths}
Let $\mathcal{C}=\{C_1,\dots,C_k\}$ be a set of vertex-disjoint cycles and let $f_1,f_2 \in  E(\mathcal{C})$ be disjoint edges.
Then $E(\mathcal{C})\setminus \{f_1,f_2\}$ can be odd-covered using 2 paths.
\end{lemma}

\begin{proof}
Observe that $E(\mathcal{C})\setminus \{f_1,f_2\}$ consists of a vertex-disjoint collection of two paths, say $R_1,R_2$, and $k'$ cycles, where $k'\in\{k-1,k-2\}$ (depending on whether $f_1,f_2$ belong to the same cycle). 
For $i\in[2]$ let $f_i'$ denote the edge with the same endpoints as $R_i$. By Lemma \ref{lem:disjointcyclesplustwoedgestwopaths}, the union of $\{f_1',f_2'\}$ and the remaining $k'$ cycles can be odd-covered by two paths, say $P'$ and $Q'$. By replacing the edges $f_1',f_2'$ by the paths $R_1,R_2$ respectively in $P'$ and $Q'$, we obtain an odd-cover of $E(\mathcal{C})\setminus\{f_1,f_2\}$ by two paths.
\end{proof}

Finally, we consider the case where one of the edges, $f_1$, lies in the set of cycles, and the other, $f_2$, does not.

\begin{lemma}\label{lem:disjointcyclesplusminusedgetwopaths}
Let $\mathcal{C}=\{C_1,\dots,C_k\}$ be a set of vertex-disjoint cycles and let $f_1\in  E(\mathcal{C})$ and $f_2\not\in  E(\mathcal{C})$ be disjoint edges.
Then $( E(\mathcal{C}) \setminus\{f_1\} )\cup \{f_2\}$ can be odd-covered using 2 paths.
\end{lemma}
\begin{proof}
Let us assume without loss of generality that $f_1\in E(C_1)$, and let $S$ denote the path $C_1\setminus f_1$ from $u_1$ to $v_1$.

If neither endpoint of $f_2$ is in $V(C_1)$, then by Lemma \ref{lem:disjointcyclesplusedgeandcommonendpointtwopaths}, the edge set $\{f_2\}\cup \bigcup_{i=2}^k E(C_i)$ can be odd-covered using two paths $P$ and $Q$ with $V(P)\cup V(Q)\subseteq \{u_2,v_2\}\cup\bigcup_{i=2}^k V(C_i)$ and with a common endpoint, say $z$. Then the two paths $P\cup zu_1 \cup S$ and $Q\cup zu_1$ odd-cover $( E(\mathcal{C})\setminus\{f_1\} )\cup \{f_2\}$.

So suppose $u_2\in V(C_1)$.
Since $f_1, f_2$ are disjoint from each other, $u_2$ is an internal vertex of $S$.
Note that we have either $v_2\in V(C_1)\setminus \{u_1,v_1\}$, $v_2 \in \bigcup_{i=2}^k V(C_i)$, or $v_2\not\in  V(\mathcal{C})$.
If $k=1$, then clearly $S\cup \{f_1\}$ can be decomposed into two paths, hence odd-covered using two paths. For the remainder of the proof we assume $k\geq 2$.

By Lemma \ref{lem:disjointcyclestwopaths}, $\bigcup_{i=2}^k E(C_i)$ can be odd-covered using two paths $P$ and $Q$ with $V(P)\cup V(Q) \subseteq \bigcup_{i=2}^k V(C_i)$. Let $z_1$ and $z_2$ denote the two common endpoints of $P$ and $Q$. If $v_2\in \bigcup_{i=2}^k V(C_i)$, then we choose $P$ and $Q$ so that $z_2=v_2$; this is always possible because the endpoints of the paths are arbitrarily chosen in the proof of Lemma \ref{lem:disjointcyclestwopaths}.

Suppose $v_2\in V(C_1)$. We may assume without loss of generality that $u_1,u_2,v_2,v_1$ occur in this order on $S$.
Then the two paths
\begin{align*}
    u_2Sv_1\ \cup\ &v_1z_1 \cup P \cup z_2u_1 \text{ and} \\
    &v_1z_1 \cup Q \cup z_2u_1 \cup u_1Su_2 \cup u_2v_2
\end{align*}
odd-cover $(E(\mathcal{C}) \setminus\{f_1\} )\cup \{f_2\}$. See Figure \ref{fig:case:-f1+f2a}.

If $v_2\in \bigcup_{i=2}^k V(C_i)$ (hence $z_2=v_2$), then the two paths
\begin{align*}
    S\ \cup\ &v_1z_1 \cup P \text{ and} \\
    &v_1z_1 \cup Q\cup v_2u_2 
\end{align*}
odd-cover $(E(\mathcal{C}) \setminus\{f_1\} )\cup \{f_2\}$. See Figure \ref{fig:case:-f1+f2b}.

Finally, if $v_2\not\in  V(\mathcal{C})$, then the two paths
\begin{align*}
    S\ \cup\ &v_1z_1 \cup P \cup z_2v_2 \text{ and} \\
    &v_1z_1 \cup Q \cup z_2v_2\cup v_2u_2
\end{align*}
odd-cover $( E(\mathcal{C}) \setminus\{f_1\} )\cup \{f_2\}$. See Figure \ref{fig:case:-f1+f2c}.
\end{proof}

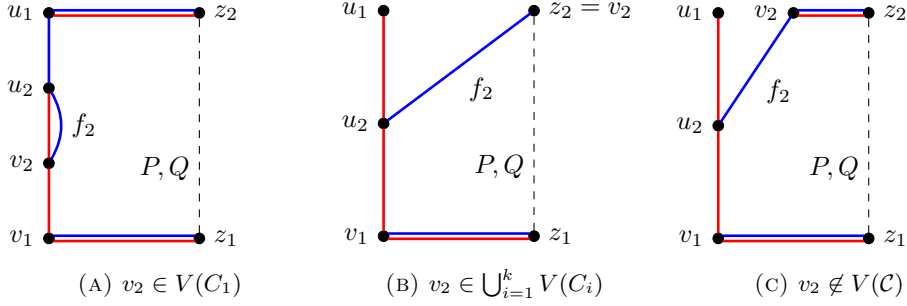
\begin{figure}[t]
\centering
\begin{subfigure}{0.31\textwidth}
\begin{tikzpicture}[vertices/.style={draw, fill=black, circle, inner sep=0pt, minimum size = 4pt, outer sep=0pt}, r_edge/.style={draw=red,line width= 1,>=latex,red}, b_edge/.style={draw=blue,line width= 1,>=latex,blue}, k_edge/.style={draw=black,line width= 1,>=latex,black}, scale=2]

\node[vertices, label=left:$v_1$] (v_1) at (0,0) {};
\node[vertices, label=left:$v_2$] (v_2) at (0,0.5) {};
\node[vertices, label=left:$u_2$] (u_2) at (0,1) {};
\node[vertices, label=left:$u_1$] (u_1) at (0,1.5) {};

\node[vertices, label=right:$z_1$] (z_1) at (1,0) {};
\node[vertices, label=right:$z_2$] (z_2) at (1,1.5) {};
\draw[dashed] (z_1) edge node[pos=0.3,left]{$P,Q$} (z_2);

\path[b_edge] (u_1) edge (u_2);
\path[b_edge] (u_2) edge[bend left=30, looseness=1] node[right]{\color{black}$f_2$} (v_2);
\path[r_edge] (u_2) edge (v_2);
\path[r_edge] (v_2) edge (v_1);

\path [b_edge] (v_1) edge[bend left=30, looseness=00,looseness=0] (z_1);
\path [r_edge] (v_1) edge[bend right=30, looseness=00,looseness=0] (z_1);
\path [b_edge] (u_1) edge[bend left=30, looseness=00,looseness=0] (z_2);
\path [r_edge] (u_1) edge[bend right=30, looseness=00,looseness=0] (z_2);

\end{tikzpicture}
\caption{$v_2\in V(C_1)$} 
\label{fig:case:-f1+f2a}
\end{subfigure}
\begin{subfigure}{0.31\textwidth}
\begin{tikzpicture}[vertices/.style={draw, fill=black, circle, inner sep=0pt, minimum size = 4pt, outer sep=0pt}, r_edge/.style={draw=red,line width= 1,>=latex,red}, b_edge/.style={draw=blue,line width= 1,>=latex,blue}, k_edge/.style={draw=black,line width= 1,>=latex,black}, scale=2]

\node[vertices, label=left:$v_1$] (v_1) at (0,0) {};
\node[vertices, label=left:$u_2$] (u_2) at (0,0.75) {};
\node[vertices, label=left:$u_1$] (u_1) at (0,1.5) {};

\node[vertices, label=right:$z_1$] (z_1) at (1,0) {};
\node[vertices, label=right:{$z_2=v_2$}] (z_2) at (1,1.5) {};
\draw[dashed] (z_1) edge node[pos=0.3,left]{$P,Q$}(z_2);

\path[r_edge] (u_1) edge (u_2);
\path[r_edge] (u_2) edge (v_1);
\path[b_edge] (z_2) edge node[below right]{\color{black}$f_2$} (u_2);

\path [b_edge] (v_1) edge[bend left=30, looseness=0] (z_1);
\path [r_edge] (v_1) edge[bend right=30, looseness=0] (z_1);

\end{tikzpicture}
\caption{$v_2\in \bigcup_{i=1}^k V(C_i)$} 
\label{fig:case:-f1+f2b}
\end{subfigure}
\begin{subfigure}{0.31\textwidth}
\begin{tikzpicture}[vertices/.style={draw, fill=black, circle, inner sep=0pt, minimum size = 4pt, outer sep=0pt}, r_edge/.style={draw=red,line width= 1,>=latex,red}, b_edge/.style={draw=blue,line width= 1,>=latex,blue}, k_edge/.style={draw=black,line width= 1,>=latex,black}, scale=2]

\node[vertices, label=left:$v_1$] (v_1) at (0,0) {};
\node[vertices, label=left:$u_2$] (u_2) at (0,0.75) {};
\node[vertices, label=left:$u_1$] (u_1) at (0,1.5) {};
\node[vertices, label=left:$v_2$] (v_2) at (0.5,1.5) {};

\node[vertices, label=right:$z_1$] (z_1) at (1,0) {};
\node[vertices, label=right:{$z_2$}] (z_2) at (1,1.5) {};
\draw[dashed] (z_1) edge node[pos=0.3,left]{$P,Q$}(z_2);

\path[r_edge] (u_1) edge (u_2);
\path[r_edge] (u_2) edge (v_1);
\path[b_edge] (v_2) edge node[below right]{\color{black}$f_2$} (u_2);

\path [b_edge] (v_1) edge[bend left=30, looseness=0] (z_1);
\path [r_edge] (v_1) edge[bend right=30, looseness=0] (z_1);
\path [b_edge] (v_2) edge[bend left=30, looseness=0] (z_2);
\path [r_edge] (v_2) edge[bend right=30, looseness=0] (z_2);

\end{tikzpicture}
\caption{$v_2\not\in V(\mathcal{C})$} 
\label{fig:case:-f1+f2c}
\end{subfigure}

\caption{The three cases in the proof of Lemma \ref{lem:disjointcyclesplusminusedgetwopaths}.} 
\label{fig:case:-f1+f2}
\end{figure}

Combining, Lemmas \ref{lem:disjointcyclesplustwoedgestwopaths}, \ref{lem:disjointcyclesminustwoedgestwopaths}, and \ref{lem:disjointcyclesplusminusedgetwopaths}, we conclude that any two edges from $M_{\odd}$ can be integrated into two paths odd-covering a set of vertex-disjoint cycles, unless the technical conditions in Lemma \ref{lem:disjointcyclesplustwoedgestwopaths} apply.

\begin{cor}\label{cor:cyclessymdifftwoedges}
    Let $\mathcal{C}=\{C_1,\dots,C_k\}$ be a set of vertex-disjoint cycles and let $f_1,f_2$ be disjoint edges. Then $\{f_1,f_2\}\oplus E(\mathcal{C})$ can be odd-covered using two paths, unless for some $j\in[k]$, $C_j\cup \{f_1,f_2\}$ is a $K_4$-subdivision and at most one of the four subpaths of $C_j$ demarcated by the endpoints of $f_1,f_2$ has an internal vertex.
\end{cor}

It remains to address this technical condition. We do so by showing that as long as there are three edges $f_1,f_2,f_3 \in M_{\odd}$, we can {\em choose} a pair to use in a two-path odd-cover (because it does not satisfy the technical conditions). In fact, out of the three unordered pairs of edges in $\{f_1,f_2,f_3\}$, at most one pair can form the exceptional case.

\begin{lemma}
\label{lem:disjointcycleschoosetwoedges}
Let $\mathcal{C}$ be a set of vertex-disjoint cycles, and let $f_1,f_2,f_3$ be disjoint edges. 
Then for at least two unordered pairs $\{f_a,f_b\}\subseteq \{f_1,f_2,f_3\}$, we have that $\{f_a,f_b\}\oplus E(\mathcal{C})$ can be odd-covered using two paths.
\end{lemma}
\begin{proof}
Let $a,b\in[3]$ and suppose that $\{f_a,f_b\}\oplus E(\mathcal{C})$ cannot be odd-covered using 2 paths. Let $f_c\in \{f_1,f_2,f_3\}\setminus \{f_a,f_b\}$.
By Corollary \ref{cor:cyclessymdifftwoedges}, we have for some $C\in \mathcal{C}$ that $C \cup \{f_a,f_b\}$ is a subdivision of $K_4$ with at most one of the four subpaths of $C$ demarcated by the endpoints of $f_a,f_b$ having an internal vertex.
It then follows that neither $C \cup \{f_a,f_c\}$ nor $C \cup \{f_b,f_c\}$ forms a $K_4$-subdivision. Hence, by Corollary \ref{cor:cyclessymdifftwoedges}, both $\{f_a,f_c\}\oplus E(\mathcal{C})$ and $\{f_b,f_c\}\oplus E(\mathcal{C})$ can be odd-covered using two paths.
\end{proof}

If there are exactly four edges $f_1,f_2,f_3,f_4$ in $M_{\odd}$, then we can integrate two at a time to two sets $\mathcal{C},\mathcal{D}$ of vertex-disjoint cycles. However, if for example we integrate $f_1,f_2$ into two paths odd-covering $\{f_1,f_2\}\oplus E(\mathcal{C})$, then $f_3,f_4$, and $\mathcal{D}$ may form an exceptional case. This can be avoided by choosing a different pair of edges to integrate into $\mathcal{C}$.

\begin{lemma} \label{lem:fouredgestworounds4paths}
Let $\mathcal{C}$ and $\mathcal{D}$ be sets of vertex-disjoint cycles such that $E(\mathcal{C})\cap E(\mathcal{D})=\emptyset$, and let $f_1,f_2,f_3,f_4$ be disjoint edges. 
Then $\{f_1,f_2,f_3,f_4\}\oplus E(\mathcal{C}) \oplus E(\mathcal{D})$ can be odd-covered using 4 paths.
\end{lemma}

\begin{proof}
    By Lemma \ref{lem:disjointcycleschoosetwoedges} applied to $\mathcal{C}$ and $f_1,f_2,f_3$, we may assume without loss of generality that both  $\{f_1,f_2\}\oplus E(\mathcal{C})$ and $\{f_1,f_3\}\oplus E(\mathcal{C})$ can be odd-covered using two paths. 
    By Lemma \ref{lem:disjointcycleschoosetwoedges} applied to $\mathcal{D}$ and $f_2,f_3,f_4$, we have for some $f_a\in\{f_2,f_3\}$ that $\{f_a,f_4\}\oplus E(\mathcal{D})$ can be odd-covered using two paths. 
\end{proof}

\subsubsection{An iterative odd-covering procedure.}

We are ready to tie our results together. First we prove a bound roughly of the form $p_2(G) \leq \frac{v_\odd(G)}{2} + 2$ for certain graphs $G$ with $\Delta(G) \leq \frac{v_\odd(G)}{2}$.

\begin{lem}\label{lem:Eulerianplusmatching}
    Let $G$ be subgraph of $K_n$ with edge set $E(G) = M \oplus E(G')$, where $M \subseteq E(K_n)$ is a matching and $G'$ is an Eulerian graph of maximum degree $\Delta(G') \leq 2t$, where
    \[t := \left\{\begin{array}{c l}
\left \lceil\frac{|M|}{2} \right \rceil & \text{if } |M| \neq 2;\\
 2 & \text{if } |M| = 2.
\end{array}\right.\]
Then $p_2(G) \leq 2t$.
    
\end{lem}
\begin{proof}
By a repeated application of Lemma \ref{lem:disjointcyclesmaxdegvertices}, we obtain a sequence $\mathcal{C}_1,\dots, \mathcal{C}_{t}$ of (possibly empty) sets of vertex-disjoint cycles which together form a partition of $E(G')$.

We iteratively find two paths that odd-cover $E(\mathcal{C}_i)$ and integrate up to two edges of $M$ as follows.
If $|M|$ is odd, then we repeatedly apply Lemma \ref{lem:disjointcycleschoosetwoedges} until we are left with a single edge, which we integrate using Corollary \ref{cor:cyclesoplusoneedge}.
If $|M|$ is even and $|M|\geq 4$, then we again apply Lemma \ref{lem:disjointcycleschoosetwoedges} repeatedly to integrate two edges at a time until we are left with exactly four edges, which we integrate into two sets of vertex-disjoint cycles using Lemma \ref{lem:fouredgestworounds4paths}.
If $|M|=2$, then we integrate one edge at a time into two sets of vertex-disjoint cycles using Corollary \ref{cor:cyclesoplusoneedge}. If $|M| = 0$, then there's nothing to prove.

More precisely, set $i=1$ and $M_1=M$. 
We repeatedly apply one of the following steps as long as $i\leq t$:
\begin{itemize}
    \item If $|M_i|\geq 5$ or $|M_i|=3$, then apply Lemma \ref{lem:disjointcycleschoosetwoedges} to obtain two paths $P_i,Q_i$ such that $P_i\oplus  Q_i = \{f_a,f_b\}\oplus  E(\mathcal{C}_i)$ for some $f_a,f_b\in M_i$. Update $M_{i+1}=M_i\setminus\{f_a,f_b\}$ and $i=i+1$.
    \item If $|M_i|=4$, then we have $i\leq t-1$. Apply Lemma \ref{lem:fouredgestworounds4paths} to find four paths $P_i,Q_i,P_{i+1},Q_{i+1}$ such that $$P_i\oplus  Q_i\oplus  P_{i+1}\oplus  Q_{i+1} = M_i \oplus  E(\mathcal{C}_i) \oplus  E(\mathcal{C}_{i+1}).$$
    Update $M_{i+2}=\emptyset$ and $i=i+2$.
    \item If $|M_i|=1$, then apply Corollary \ref{cor:cyclesoplusoneedge} to obtain two paths $P_i,Q_i$ such that $P_i\oplus Q_i = M_i \oplus E(\mathcal{C}_i)$. Update $M_{i+1}=\emptyset$ and $i=i+1$.
    \item If $|M_i|=2$, then observe that $i=1$ since none of the above steps leaves $M_i$ with exactly two edges. This implies $|M| = t = 2$. Apply Corollary \ref{cor:cyclesoplusoneedge} to $(f_1, \mathcal{C}_1)$ and to $(f_2,\mathcal{C}_2)$ to obtain four paths $P_1,Q_1,P_2,Q_2$ such that $P_1\oplus Q_1 = \{f_1\}\oplus E(\mathcal{C}_1)$ and $P_2\oplus Q_2=\{f_2\}\oplus E(\mathcal{C}_2)$. Update $M_{i+2}=\emptyset$ and $i=i+2$.
    \item If $M_i=\emptyset$, then we already have $i > t$.
    
\end{itemize}
After this procedure, we obtain a sequence of paths $P_1,Q_1,P_2,Q_2,\dots,P_{t},Q_{t}$ such that 
\begin{align*}
    \bigoplus_{i=1}^{t} (P_i\oplus Q_i) = M\oplus \bigcup_{i=1}^{t} E(\mathcal{C}_i) = M\oplus E(G') =  E(G).
\end{align*}
Hence, $\{P_1,Q_1,\dots,P_{t},Q_{t}\}$ is a set of $2t$ paths that odd-cover $E(G)$.
\end{proof}

\subsubsection{Improved bounds on $p_2(G)$ and $p_{2,\iso}(G)$}

We are now ready to prove Theorem \ref{thm:upperboundDeltaVSo2} claiming the existence of an odd-cover of at most $\max\{\frac{v_{\odd}}{2}, 2\left\lceil \frac{\Delta}{2} \right\rceil\}$ paths.

\DeltaVSodd*

\begin{proof}[Proof of Theorem \ref{thm:upperboundDeltaVSo2}]
Let $\Delta_e=2\left\lceil \frac{\Delta(G)}{2} \right\rceil$.
First suppose that $\frac{v_{\odd}}{2}\leq \Delta_e$. Let $M_{\odd}$ be an arbitrary perfect matching of the odd-degree vertices (in the underlying complete graph) and let $G'$ be the graph on the vertex set $V(G)$ with edge set $E(G')= E(G) \oplus M_\odd$. Then $G'$ is a simple Eulerian graph of maximum degree at most $\Delta_e$. Note that $|M_\odd|=\frac{v_{\odd}}{2}$.

By repeated application of Lemma \ref{lem:disjointcyclesmaxdegvertices}, we obtain a sequence $\mathcal{C}_1,\dots, \mathcal{C}_{\frac{\Delta_e}{2}}$ of sets of vertex-disjoint cycles which together form a parition of $E(G')$.

If $\Delta(G) \leq 2$ and $|M_\odd| = 2$, then $\Delta(G') \leq 3$. Since $G'$ is Eulerian, this means that $G'$ is a union of vertex-disjoint cycles, so we have $p_2(G) \leq 2 = \Delta_e$ by Corollary \ref{cor:cyclessymdifftwoedges}.

Otherwise, let $G''$ be the Eulerian graph with edge set $E(G'') = \bigcup_{i = 1}^t E(\mathcal C_i),$ where
\[t := \left\{\begin{array}{c l}
\left \lceil\frac{|M_\odd|}{2} \right \rceil & \text{if } |M_\odd| \neq 2;\\
2 & \text{if } |M_\odd| = 2.
\end{array}\right.\]
Note that $t \leq \frac{\Delta_e}{2}$. Indeed, if $|M_\odd| = 2$, then since we've already dealt with the case $\Delta(G) \leq 2$, we must have $t = 2 \leq \frac{\Delta_e}{2}$. If $|M_\odd| \neq 2$, then $|M_\odd| = \frac{v_\odd}{2} \leq \Delta_e$ implies $t = \left \lceil \frac{|M_\odd|}{2} \right \rceil \leq \frac{\Delta_e}{2}$ since $\Delta_e$ is even. Now we apply Lemma \ref{lem:Eulerianplusmatching} to odd-cover $E(G'') \oplus M_\odd$ with at most $2t$ paths, and we apply Lemma \ref{lem:disjointcyclestwopaths} to odd-cover each $\mathcal C_i$ with $2$ paths for $t+1 \leq i \leq \frac{\Delta_e}{2}$. Altogether, we odd-cover $E(G) = E(G'') \oplus \bigcup_{i = t+1}^{\frac{\Delta_e}{2}} E(\mathcal C_i)$ with at most $2t + 2\left(\frac{\Delta_e}{2} - t\right) = \Delta_e$ paths.

If $\frac{v_{\odd}}{2}> \Delta_e$, then we arbitrarily select a path in $G$ connecting odd-degree vertices and delete its edges. We repeat until $\frac{v_{\odd}}{2}\leq \Delta_e$. Note that this requires $\frac{v_{\odd}}{2}- \Delta_e$ steps. We then apply the above argument to odd-cover the remaining edges with $\Delta_e$ paths. This requires a total of $\Delta_e+(\frac{v_{\odd}}{2} - \Delta_e) = \frac{v_{\odd}}{2}$ paths.
\end{proof}

We now present a proof of Theorem \ref{thm:isolatedgeneral} as well, which improves the bound of Theorem \ref{thm:upperboundDeltaVSo2} if we allow for addition of isolated vertices. We restate the theorem here.

\IsolatedGeneral*

\begin{proof}[Proof of Theorem \ref{thm:isolatedgeneral}]
    Let $\Delta_e=2\left\lceil \frac{\Delta(G)}{2} \right\rceil$ and $v_\odd = v_\odd(G)$. 
    
    If $v_\odd = 4$ and $\Delta(G) \leq 2$, then $t = 1$, $\Delta_e = 2$, and $d = 0$. In fact, $G$ is a vertex-disjoint union of cycles and paths, with two path components, so by Lemma \ref{lem:disjointcyclesminustwoedgestwopaths}, $p_2(G) \leq 2 = 2t + \la(d)$.
    
Otherwise, suppose that $d \geq 0$, and so $2t \leq \Delta_e$. Consider a perfect matching $M_{\odd}$ of the odd-degree vertices (in the underlying complete graph) and the Eulerian graph $G'$ on $V(G)$ with edge set $E(G')= E(G) \oplus M_\odd$, as in the proof of Theorem \ref{thm:upperboundDeltaVSo2}.

By repeated application of Lemma \ref{lem:disjointcyclesmaxdegvertices}, we can decompose $E(G')$ into edge-disjoint Eulerian graphs $G''$ and $G'''$ with respective maximum degrees at most $2t$ and $\Delta_e - 2t=d$.
Now by Lemma \ref{lem:Eulerianplusmatching}, the edges from $E(G'') \oplus M_{\odd}$ can be odd-covered by $2t$ paths,
and by Theorem \ref{thm:p2vsla}, $E(G''')$ can be odd-covered by $\la(G''') \leq \la(d)$ paths, if we add a sufficient number of isolated vertices. Thus $p_{2,\iso}(G) \leq 2t + \la(d)$.

If $d < 0$, then we can't be in the case that $v_\odd = 4$ and $\Delta(G) \geq 3$, so we must have $d = \Delta_e - 2\left \lceil \frac{v_\odd}{4} \right \rceil < 0$. Now because $\Delta_e$ is even, this implies $\frac{v_{\odd}}{2} \geq \Delta_e$. Then by Theorem \ref{thm:upperboundDeltaVSo2}, $p_{2,\iso}(G) \leq p_2(G) \leq \frac{v_\odd}{2}$.
\end{proof}

\section{Cycle odd-covers}
\label{sec:cycleOdd}
We now adapt our methods to cycle odd-covers to prove Theorem~\ref{thm:cycleOddCover}, which gives upper bounds for the cycle odd-cover number $c_2(G)$ and its relaxations $c_{2,\topo}(G)$ and $c_{2,\iso}(G)$. Since $\Delta(G)/2$ is a lower bound for each of these parameters, the value of $c_{2,\topo}(G)$ is entirely determined, $c_{2,\iso}(G)$ is determined up to a $1 + o(1)$ factor, and $c_2(G)$ is determined up to a factor of $2$. This complements the work of Fan, whose main result from \cite{fan2003covers} implies that $c_2(G) \leq \lfloor (n-1)/2 \rfloor$ for any Eulerian graph $G$ on $n$ vertices.
\CycleOddCover*
\begin{proof}
    We first consider (1). We begin by applying the proof of Theorem \ref{thm:topological} to obtain, for some subdivision $G'$ of $G$, a well-distributed path $k$-system $\{\mathcal P_1, \ldots, \mathcal P_k\}$ with $k = \frac{\Delta(G)}{2}$, where each $\mathcal P_i$ consists of a single path $P_i$, and $E(G') = \bigoplus_{i = 1}^k P_i$. We choose any of these paths, say $P_1$, and we consider its endpoints $u$ and $v$ and the respective paths $P_i$ and $P_j$ that meet $P_1$ at $u$ and $v$. If $i = j$, then applying the operation $\join(u,v)$ converts $P_1$ and $P_i$ into cycles, which we can remove and obtain a path $(k-2)$-system. If $i \neq j$, then we can apply $\ins(u,\mathcal P_i, \mathcal P_j)$ and $\join(u,v)$ successively. This converts $P_1$ into a cycle and extends $P_j$ into a longer path that now meets $P_i$ at the new subdivided vertex. We can remove $P_1$ and obtain a path $(k-1)$-system. Iterating this process at most $k$ times, we obtain a cycle odd-cover of some subdivision of $G'$ with $k$ cycles.

    We now consider (2). We begin by adding isolated vertices to $G$ and obtaining a path odd-cover $E(G) = \bigoplus_{i = 1}^k P_i$ by Theorem \ref{thm:p2vsla}, where $k = \la(G)$. We denote the endpoints of $P_i$ by $u_i$ and $v_i$. We add one more isolated vertex $w$, and we define the cycle $C_i := P_i \cup \{u_iw, v_iw\}$ for each $i$. Since $G$ is Eulerian, every vertex of $G$ appears an even number of times as endpoints from $P_1, \ldots, P_k$, so 
    \[E(G) = \bigoplus_{i = 1}^k P_i = \bigoplus_{i = 1}^k (P_i \oplus \{u_iw, v_iw\}) = \bigoplus_{i = 1}^k C_i,\]
    giving the desired cycle odd-cover.

    Finally, we consider (3). By repeated application of Lemma \ref{lem:disjointcyclesmaxdegvertices}, we obtain a sequence $\mathcal{C}_1,\dots, \mathcal{C}_{\frac{\Delta(G)}{2}}$ of sets of vertex-disjoint cycles which together form a partition of $E(G)$. Each $\mathcal C_i$ can be odd-covered by two paths, say $P_i$ and $Q_i$. Since $\mathcal C_i$ is 2-regular, these paths must share endpoints, say $u_i$ and $v_i$. Adding the edge $u_iv_i$ to $P_i$ and to $Q_i$ gives cycles $C_i$ and $D_i$ with
    \[C_i \oplus D_i = P_i \oplus Q_i = E(\mathcal C_i).\]
    Altogether, $\{C_i, D_i\}_{i = 1}^\frac{\Delta(G)}{2}$ forms a cycle odd-cover of $G$ with $\Delta(G)$ cycles.
\end{proof}

\section{Future work and open problems}\label{sec:openquestions}

This problem can be generalized to odd-covers of graphs by any family $\mathcal{F}$ of graphs instead of paths and cycles. Clique odd-covers were studied in~\cite{buchanan2022subgraph}, while biclique odd-covers were studied in~\cite{buchanan2022odd}. The proof techniques in~\cite{buchanan2022odd,buchanan2022subgraph} include minimum rank arguments, vertex covers, linear algebra, and forbidden subgraphs. They are quite different from the techniques used in this paper, highlighting that different approaches can be successful for odd-cover problems depending on the family of covering graphs considered.

While the gap between the rather immediate lower bound of Equation~\ref{eq:vodd,maxdegree_lower} and the upper bound  of Theorem~\ref{thm:upperboundDeltaVSo2} is a constant factor, we have yet to find an example of a graph with path odd-cover number significantly greater than this lower bound. This leads to the following question.

\begin{prob}
    \label{prob:oneMoreLower}
    Does there exist a graph $G$ with
    \[ p_2(G)>\max\left\{\frac{v_\odd(G)}{2}, \left\lceil\frac{\Delta(G)+1}{2}\right\rceil \right\} \text{?} \]
\end{prob}

This problem is related to the linear arboricity conjecture of Akiyama, Exoo, and Harary \cite{Akiyama1981}, which states that the linear arboricity of a graph of maximum degree $d$ is at most $\lceil \frac{d+1}{2} \rceil$.

If a counterexample $G$ to the linear arboricity conjecture existed, then:
\[ p_2(G) \ge \la(G) > \left\lceil \frac{\Delta(G) + 1}{2} \right\rceil \]
If it is the case that $v_\odd(G) \le \Delta(G) + 1$, then this would imply a positive answer to Problem~\ref{prob:oneMoreLower}. 
If instead it is the case that $G$ has even maximum degree $\Delta$, then we can embed $G$ in a $\Delta$-regular (Eulerian) graph $G'$, so
\[ p_2(G') \ge \la(G') \ge \la(G) > \left\lceil \frac{\Delta + 1}{2} \right\rceil,\]
giving a positive answer to Problem~\ref{prob:oneMoreLower}.

Stated otherwise, a negative answer to Problem~\ref{prob:oneMoreLower} would imply the linear arboricity conjecture for all graphs $G$ with $v_{\odd}(G) \leq \Delta(G)+1$ and graphs with even maximum degree.

We pose a similar question for cycle odd-covers.
\begin{prob}
    Does there exist an Eulerian graph $G$ with
    \[c_2(G) > \frac{\Delta(G)}{2}+1 \text{?}\]
\end{prob}

Despite considerable effort, Gallai's famous conjecture that the edge set of every $n$-vertex graph can be decomposed into at most $\lceil n/2 \rceil$ edge-disjoint paths~\cite{lovasz1968covering} still remains open. For our weakening of the problem, new tools became available which allowed us to provide an affirmative answer in our setting for graphs with maximum degree at most $n/2$. We leave the general case as a conjecture of our own.

\begin{conj}
    \label{conj:Gallai}
    For all $n$-vertex graphs $G$, 
    \[ p_2(G) \le \left \lceil \frac{n}{2}\right \rceil. \]
\end{conj}

This conjecture is also supported by the result of Fan \cite{fan2003covers} that $c_2(G) \leq \lfloor (n-1)/2 \rfloor$ for every Eulerian $n$-vertex graph $G$. It would be interesting to consider whether a similar argument could be used for the path odd-cover setting.

If we have a graph $G$ of maximum degree $3$, then it has linear arboricity at most two~\cite{Akiyama1981}. If $G$ is also Eulerian, then by Theorem~\ref{thm:p2vsla} it can be path odd-covered with at most two paths. It has been shown that graphs of maximum degree $4$ have linear arboricity at most three~\cite{Akiyama1981}. Analogous to Theorem~\ref{thm:p2vsla}, can we show Eulerian graphs of this type can be path odd-covered with at most three paths?

\begin{prob}
    \label{prob:Eulerian4}
    For all Eulerian graphs $G$ of maximum degree $4$, is it the case that
    \[ p_2(G) \le 3 \text{?} \]
\end{prob}

A positive answer to this problem would improve the upper bound of Theorem \ref{thm:upperboundDeltaVSo2} from approximately $\Delta(G)$ to $\frac{3}{4}\Delta(G)$ for graphs with few odd-degree vertices. 
Similar bounds for higher maximum degrees would improve the upper bound even further.

\section*{\small Acknowledgments}
{\small 
This work was completed in part at the {\em 2022 Graduate Research Workshop in Combinatorics}, which was supported in part by NSF grant 1953985, and by a generous award from the Combinatorics Foundation. We would like to thank Stacie Baumann, Matthias Brugger, Connor Mattes, and Garrett Nelson for insightful discussions.

Borgwardt gratefully acknowledges support by the National Science Foundation, Algorithmic Foundations, Division of Computing and Communication Foundations, under grant 2006183 {\em Circuit Walks in Optimization}, and by the Airforce Office of Scientific Research under grant FA9550-21-1-0233 {\em The Hirsch Conjecture for Totally-Unimodular Polyhedra}.}

\end{document}